\documentclass[11pt]{article}

\usepackage[T1]{fontenc}
\usepackage{lmodern}
\usepackage[utf8]{inputenc}
\usepackage{microtype}
\usepackage{framed}
\usepackage{listings}
\usepackage{vmargin}
\usepackage{setspace}
\usepackage{mathrsfs, mathenv}
\usepackage{amsmath, amsthm, amssymb, amsfonts, amscd}
\usepackage{graphicx}
\usepackage{epstopdf}
\usepackage[svgnames]{xcolor}
\usepackage{hyperref}
\hypersetup{citecolor=blue, colorlinks=true, linkcolor=black}
\usepackage[ruled, vlined]{algorithm2e}
\usepackage[capitalise]{cleveref}
\usepackage{subcaption}
\usepackage{bbm}
\usepackage{cite}
\usepackage{verbatim}
\usepackage{pgfplots}
\usepackage{tikz}
\usetikzlibrary{arrows,decorations.pathmorphing,backgrounds,positioning,fit,matrix}
\usepackage{etoolbox}
\usepackage{color}
\usepackage{lipsum}
\usepackage{ifthen}
\usepackage[title]{appendix}
\usepackage{autonum}
\usepackage{accents}
\usepackage{xpatch}
\usepackage[]{algpseudocode}
\usepackage{multicol}
\usepackage[abs]{overpic}
\usepackage{eqnarray}
\usepackage{diagbox}

\crefname{assumption}{Assumption}{Assumptions}
\crefname{figure}{Figure}{Figures}

\theoremstyle{plain}
\newtheorem{theorem}{Theorem}[section]
\newtheorem{corollary}[theorem]{Corollary}
\newtheorem{lemma}[theorem]{Lemma}
\newtheorem{proposition}[theorem]{Proposition}
\numberwithin{equation}{section}

\theoremstyle{definition}

\theoremstyle{remark}
\newtheorem{remark}[theorem]{Remark}

\newtheorem{example}[theorem]{Example}

\ifpdf
  \DeclareGraphicsExtensions{.eps,.pdf,.png,.jpg}
\else
  \DeclareGraphicsExtensions{.eps}
\fi

\usepackage{mathtools}

\usepackage{booktabs}

\usepackage{pgfplots}
\usepackage{tikz}
\usetikzlibrary{patterns,arrows,decorations.pathmorphing,backgrounds,positioning,fit,matrix}
\usepackage[labelfont=bf]{caption}
\usepackage{here}
\usepackage[font=normal]{subcaption}

\usepackage{enumitem}
\setlist[itemize]{leftmargin=.5in}
\setlist[enumerate]{leftmargin=.5in,topsep=3pt,itemsep=3pt,label=(\roman*)}

\newcommand{\email}[1]{\href{#1}{#1}}
\newcommand{\TheTitle}{Enabling stratified sampling in high dimensions \\ via nonlinear dimensionality reduction}
\newcommand{\TheAuthors}{G. Geraci, D. E. Schiavazzi, A. Zanoni}
\title{\TheTitle}
\author{Gianluca Geraci \thanks{Center for Computing Research, Sandia National Laboratories, Albuquerque, NM, USA, \email{ggeraci@sandia.gov}.} \and Daniele E. Schiavazzi \thanks{Department of Applied and Computational Mathematics and Statistics, University of Notre Dame, Notre Dame, IN, USA, \email{dschiavazzi@nd.edu}.} \and Andrea Zanoni \thanks{Centro di Ricerca Matematica Ennio De Giorgi, Scuola Normale Superiore, Pisa, Italy, \email{andrea.zanoni@sns.it}.}}
\date{}

\newcommand{\Q}{\mathcal Q}
\newcommand{\E}{\mathcal E}
\newcommand{\D}{\mathcal D}
\newcommand{\F}{\mathcal F}
\renewcommand{\S}{\mathcal S}

\newcommand{\LF}{\mathrm{LF}}
\newcommand{\HF}{\mathrm{HF}}
\newcommand{\sMC}{\mathrm{sMC}}
\newcommand{\MC}{\mathrm{MC}}
\newcommand{\MFMC}{\mathrm{MFMC}}
\newcommand{\sMFMC}{\mathrm{sMFMC}}

\usepackage{amsopn}

\DeclareMathOperator{\erf}{erf}

\newcommand{\abs}[1]{\left\lvert#1\right\rvert}
\newcommand{\norm}[1]{\left\|#1\right\|}
\renewcommand{\Pr}{\mathbb{P}}

\newcommand{\R}{\mathbb{R}}

\newcommand{\epl}{\varepsilon}

\newcommand{\Var}{\operatorname{Var}}
\newcommand{\Cov}{\operatorname{\mathbb{C}ov}}
\newcommand{\Ex}{\operatorname{\mathbb{E}}}

\DeclareMathOperator*{\argmax}{arg\,max}
\DeclareMathOperator*{\argmin}{arg\,min}

\newcommand{\dd}{\,\mathrm{d}}
\definecolor{shade}{RGB}{100, 100, 100}
\definecolor{bordeaux}{RGB}{128, 0, 50}

\usepackage[usestackEOL]{stackengine}

\definecolor{leg1}{RGB}{0,114,189}
\definecolor{leg2}{RGB}{217,83,25}
\definecolor{leg3}{RGB}{237,177,32}
\definecolor{leg4}{RGB}{126,47,142}
\definecolor{leg5}{RGB}{119,172,48}

\definecolor{leg21}{RGB}{62,38,169}
\definecolor{leg22}{RGB}{46,135,247}
\definecolor{leg23}{RGB}{55,200,151}
\definecolor{leg24}{RGB}{254,195,56}

\ifpdf
\hypersetup{
	pdftitle={\TheTitle},
	pdfauthor={\TheAuthors}
}
\fi

\begin{document}
	
\maketitle	

\begin{abstract} 
\noindent We consider the problem of propagating the uncertainty from a possibly large number of random inputs through a computationally expensive model. 
Stratified sampling is a well-known variance reduction strategy, but its application, thus far, has focused on models with a limited number of inputs due to the challenges of creating uniform partitions in high dimensions.
To overcome these challenges, we propose a simple methodology for constructing an effective stratification of the input domain that is adapted to the model response. Our approach leverages neural active manifolds, a recently introduced nonlinear dimensionality reduction technique based on neural networks that identifies a one-dimensional manifold capturing most of the model variability. The resulting one-dimensional latent space is mapped to the unit interval, where stratification is performed with respect to the uniform distribution. The corresponding strata in the original input space are then recovered through the neural active manifold, generating partitions that tend to follow the level sets of the model.
We show that our approach is effective in high dimensions and can be used to further reduce the variance of multifidelity Monte Carlo estimators.
\end{abstract}



\section{Introduction}

Mathematical modeling and numerical simulations are fundamental in most scientific and engineering disciplines for advancing our ability to understand and predict complex phenomena~\cite{FoM21}. 
Yet, the predictive power of these simulations is invariably affected by our imperfect knowledge of underlying mechanisms and their inherent variability. Ultimately, understanding and quantifying the uncertainty in computational model outputs is essential for establishing their validity.
This led to increasing recent interest in the development of computational efficient strategies to propagate input uncertainty through complex computational models~\cite{Sul15}. 
Many of these strategies consider uncertain inputs as random variables, and provide approximations for the expectation, or higher order moments, of one or multiple quantities of interests (QoIs). 
Consider the computational model $\Q \colon \R^d \to \R$, and let $X \sim \mu$ be a collection of input parameters with distribution $\mu$ on $\R^d$. 
We seek to estimate the quantity
\begin{equation}\label{eq:QoI}
q = \Ex^\mu[\Q(X)],
\end{equation}
where the expectation is computed with respect to the probability measure $\mu$. 

A common approach to approximate $q$ is through Monte Carlo sampling. 
Given a set $\{ x_n \}_{n=1}^N$ of realizations from the distribution $\mu$, an estimator is defined as
\begin{equation}
\widehat q_\MC = \frac1N \sum_{n=1}^N \Q(x_n),
\end{equation}
which is unbiased, i.e., $\Ex[\widehat q_\MC] = q$, and has variance
\begin{equation}
\Var[\widehat q_\MC] = \frac1N \Var^\mu[\Q(X)].
\end{equation}
Thus, the cost of producing an estimate $\widehat q_\MC$ depends on the cost of solving the computational model $\Q$, with a precision that is directly affected by the number of samples $N$ and the variance of model $\Q$ itself. 
Therefore, obtaining accurate estimates for the statistical moments of QoIs from high-fidelity models can easily become computationally intractable.

To make this computationally feasible, it is essential to reduce the variance of these estimates. Many methods have been proposed in the literature to achieve this.
In this work, we focus on stratified sampling, which is based on a decomposition of the support of $X$ in multiple \emph{strata} of smaller variance~\cite[Chapter 5]{Coc77}. 
This approach is known to scale poorly to high dimensions, as the number of partitions needed to keep a constant number of strata in each dimension grows exponentially.
Variance reduction is also achieved by quasi-Monte Carlo estimators where random samples are replaced by deterministic low-discrepancy sequences~\cite{MoC95}, such as Halton, Hammersley, and Sobol' sequences~\cite{HaH65,Hal60,Sob67}. 
Even if quasi-Monte Carlo (qMC) can theoretically achieve asymptotically faster convergence than standard Monte Carlo under appropriate regularity assumptions on the model, its performance may also suffer in high dimensions due to challenges of generating uncorrelated samples. Furthermore,  qMC requires the use of a power-of-two number of samples to maintain optimal low-discrepancy properties.  
Dimensionality reduction techniques have been applied to quasi-Monte Carlo methods to mitigate this problem, combined with smoothness in~\cite{MoC96}, and more recently with quadratic regression in~\cite{ImT25}.
We also mention importance sampling, where samples are drawn from a different distribution than the one of interest~\cite{KlV78}, and antithetic sampling where one aims to get samples that are negatively correlated~\cite{HaM56}. 
Finally, an increasingly popular family of approaches include control variates~\cite{Lem17}, multilevel Monte Carlo~\cite{Gil15}, and multifidelity Monte Carlo, that leverage cheaper low-fidelity approximations of the expensive high-fidelity model~\cite{NgW14}.
For a comprehensive review on variance reduction techniques for Monte Carlo the interested reader is referred to \cite[Chapter 4]{Gla04}.

We propose a methodology based on nonlinear dimensionality reduction to generate partitions that are \emph{adapted} to the properties of the model $\Q$.
In practice, we employ neural active manifolds (NeurAM)~\cite{ZGS25} to determine a one-dimensional manifold that follows the variability of the model, resulting in strata that tend to be separated by the level sets of $\Q$.
NeurAM combines an autoencoder and a low-dimensional surrogate model built on a one-dimensional latent space.
In addition, through projections on the inverse cumulative distribution function, it provides an invertible transformation between the latent space and a uniform distribution supported on the unit interval. 
Thus NeurAM allows a straightforward partition on the unit interval to be mapped to corresponding strata in the original domain.
The fact that the stratification is performed in the one-dimensional unit interval is the crucial point allowing the method to scale to high-dimensional input domains.
Since NeurAM and, more generally, dimensionality reduction, have already been successfully applied to improve the performance of multifidelity Monte Carlo estimators in~\cite{ZGE23,ZGS24a,ZGS24b,MZK24}, we show how this novel stratification can also be implemented in the context of multifidelity estimators. 

The main contributions of this work are summarized below.
\begin{itemize}[leftmargin=*]
\item We introduce a scalable methodology to generate stratified sampling estimators for high-dimensional problems.
\item We show that the proposed approach shares the properties of traditional stratified sampling estimators. In particular, the estimator remains unbiased and there exist optimal allocations that guarantee variance reduction with respect to standard Monte Carlo.
\item As an alternative to uniform stratification, we provide a heuristic algorithm that, at the price of slightly increasing the computational cost, further reduces the variance of the resulting estimator.
\item We provide extensive numerical evidence that our approach is both superior to traditional stratified sampling in low dimensions and scalable to high-dimensional problems.
\item We combine NeurAM-based stratified sampling with multifidelity Monte Carlo estimators, provide conditions leading to variance reduction, and show numerically the advantages resulting from combining the two approaches.
\end{itemize}

\paragraph{Outline} This paper proceeds as follows. In \cref{sec:method}, we introduce our methodology, analyze the variance of the proposed estimator, and present a heuristic algorithm for stratification. Next, in \cref{sec:multifidelity} we apply this approach to multifidelity estimators. Then, in \cref{sec:experiments} we present numerical examples to demonstrate the properties and potential of our technique. Finally, \cref{sec:conclusion} concludes the paper and suggests avenues for future research.

\section{Stratified sampling} \label{sec:method}

Let $\mathbb D \subseteq \R^d$ be the support of the distribution $\mu$, and consider a partition $\{ D_s \}_{s=1}^S$ of $\mathbb D$ into $S$ non-overlapping strata such that
\begin{equation} \label{eq:condition_stratification}
\mathbb D = \bigcup_{s=1}^S D_s \qquad \text{and} \qquad \mu(D_i \cap D_j) = 0 \text{ if } i \neq j.
\end{equation}
Moreover, let $\{ N_s \}_{s=1}^S$ be the number of samples in each stratum such that 
\begin{equation}
N = \sum_{s=1}^S N_s,
\end{equation}
where $N$ is the available computational budget. Then, the stratified Monte Carlo estimator is defined as
\begin{equation} \label{eq:sMC_general}
\widehat q_\sMC = \sum_{s=1}^S \mu(D_s) \frac1{N_s} \sum_{n=1}^{N_s} \Q(x_n^{(s)}),
\end{equation}
where $\{ \{ x_n^{(s)} \}_{n=1}^{N_s} \}_{s=1}^S$ is the collection of samples such that $x_n^{(s)} \sim \mu |_{D_s}$, which denotes the distribution $\mu$ conditioned on the stratum $D_s$. 
We remark that the rationale behind $\widehat q_\sMC$ is the law of total expectation, leading to an estimator that satisfies $\Ex[\widehat q_\sMC] = q$. 
Moreover, its variance is given by
\begin{equation}
\Var[\widehat q_\sMC] =  \sum_{s=1}^S  \frac{\mu(D_s)^2}{N_s} \Var^\mu[\Q(X) | X \in D_s].
\end{equation}
Under an appropriate allocation $\{ N_s \}_{s=1}^S$ of the $N$ samples, which we will discuss in the next section, the variance of the stratified estimator is never larger than the variance of the corresponding standard Monte Carlo estimator with the same computational budget. 
However, the main challenge in stratified sampling is the selection of the strata $\{D_s\}_{s=1}^S$, which represents a serious obstacle for its application to high-dimensional models, since stratified sampling faces the curse of dimensionality~\cite[Lemma 3]{PeK22}.

In high-dimensional settings, one can adopt a more effective approach known as Latin Hypercube Sampling (LHS), which involves drawing samples that are stratified across each dimension \cite{MBC79}. Intuitively, LHS is the high-dimensional analogue of placing one sample in each row and each column of a regular two-dimensional grid. Although LHS improves standard stratified sampling, its effectiveness also decreases for high-dimensional models. Moreover, this technique assumes that input variables are independent, making it unsuitable for problems with correlated inputs.
A different strategy for constructing more clever strata has been proposed in \cite{EFJ11}, where the strata, defined as hyperrectangles, and their sample allocation are adaptively updated during the estimation process based on the directions that define these hyperrectangles. This approach enhances the accuracy of the estimators, particularly in the asymptotic regime where the number of samples and strata is large.
Alternative strategies are based on iterative refinements of an existing stratification. In \cite{STH15}, new samples are added sequentially by dividing existing strata along the direction identified by the unstratified component of largest magnitude.
This approach is further combined with hierarchical LHS in \cite{Shi16} to improve its performance for high-dimensional problems. 
In \cite{PeK22}, the authors propose to combine optimal and proportional sampling, in a way that is applicable to nonsmooth models, and provide theoretical estimates for the expected performance.
LHS can also be applied within each stratum to further reduce the variance of the resulting stratified estimator \cite{KrP24}. 
In all these adaptive strategies, multiple iterations are required to achieve an effective stratification, and new samples need to be generated at each iteration.

Additionally, even after selecting the strata, it is not always straightforward to compute their probability, which is a necessary step for computing $\widehat q_\sMC$.
To overcome these problems, we use nonlinear dimensionality reduction, and specifically, the NeurAM algorithm introduced in~\cite{ZGS25}, to inform the stratification of the domain.

\subsection{NeurAM-based stratification} \label{sec:NeurAM_stratification}

NeurAM aims to determine a one-dimensional manifold $\gamma$ capturing the variability of the model $\Q$ by employing an autoencoder $(\E, \D)$ that combines an encoder $\E \colon \mathbb D \to \R$ and a decoder $\D \colon \R \to \mathbb D$.
Moreover, let $\S \colon \R \to \R$ be a one-dimensional surrogate defined over the latent space of the autoencoder, and the quantities $\E,\D,\S$ are determined by minimizing a loss function of the form
\begin{equation}\label{eq:loss_function}
\begin{aligned}
\mathcal L(\E, \D, \S) &= \Ex^\mu \left[ (\Q(X) - \S(\E(\D(\E(X)))))^2 \right] + \Ex^\mu \left[ (\Q(X) - \S(\E(X)))^2 \right] \\
&\quad+ \Ex^\mu \left[ (\D(\E(X)) - \D(\E(\D(\E(X)))))^2 \right].
\end{aligned}
\end{equation} 
The three terms in the loss function play different roles. 
Specifically, the second term guarantees that the surrogate model $\S$, defined on the latent space of the autoencoder, accurately approximates the original model $\Q$, i.e., $\Q \simeq \S \circ \E$. 
The first term, which is instead essential to identify the NeurAM, enforces that projecting a point onto such one-dimensional manifold should preserve the corresponding model output, i.e., $\Q(X) \simeq Q(\widetilde X)$, where $\widetilde X = \D(\E(X))$. 
Since evaluating $\Q$ directly can be expensive, the surrogate model is used to approximate $\Q(X)$ at the projection $\widetilde X$ along the NeurAM. 
Finally, the third term ensures that if a point already lies on the manifold, then it is projected onto itself, i.e., $\D(\E(\widetilde X)) \simeq \widetilde X$, or equivalently, that the NeurAM is made up of fixed points of the composition of the encoder and decoder. Note that a global optimum leading to a zero loss can be expressed in closed form as
\begin{equation}\label{eq:ideal_solution}
\E = \Q, \qquad \Q \circ \D = \mathcal I, \qquad \S = \mathcal I,
\end{equation}
where $\mathcal I$ stands for the identity function. 
This solution is however not easy to compute, due to the possibly large computational cost of evaluating the model $\Q=\mathcal E$, and the complexity of computing $\mathcal D$ as the right inverse of $\Q$. 
Therefore, in practice, we parameterize $\widetilde\E(\cdot, \mathfrak e), \widetilde\D(\cdot, \mathfrak d), \widetilde\S(\cdot, \mathfrak s)$ as neural networks with weights $\mathfrak e, \mathfrak d, \mathfrak s$, respectively, and solve a  minimization problem where the loss function is approximated by
\begin{equation}\label{eq:loss_NeurAM_approx}
\begin{aligned}
\widetilde{\mathcal L}(\mathfrak e, \mathfrak d, \mathfrak s) &= \frac1M \sum_{m=1}^M (\Q(x_m) - \widetilde\S(\widetilde\E(\widetilde\D(\widetilde\E(x_m; \mathfrak e); \mathfrak d); \mathfrak e); \mathfrak s))^2 + \frac1M \sum_{m=1}^M (\Q(x_m) - \widetilde\S(\widetilde\E(x_m; \mathfrak e); \mathfrak s))^2 \\
&\quad + \frac1M \sum_{m=1}^M (\widetilde\D(\widetilde\E(x_m; \mathfrak e); \mathfrak d) - \widetilde\D(\widetilde\E(\widetilde\D(\widetilde\E(x_m; \mathfrak e); \mathfrak d); \mathfrak e); \mathfrak d))^2,
\end{aligned}
\end{equation}
where $\{ x_m \}_{m=1}^M$ is a set of realizations from the distribution $\mu$. 
We remark that NeurAM also automatically provides a surrogate model for $\Q$ given by $\Q_{\mathrm S} = \S \circ \E$. 
Additionally, realizations from the latent space are mapped to the unit interval $[0,1]$ as follows.
Let $\F$ be the cumulative distribution function (CDF) of the latent variable $\E(X)$, i.e.,
\begin{equation} \label{eq:CDF}
\F(t) = \Pr^\mu(\E(X) \le t).
\end{equation}
By the inverse transform sampling, it follows that 
\begin{equation} \label{eq:inverse_transform_sampling}
(\F \circ \E)_\# \mu = \mathcal U([0,1]),
\end{equation}
which means that for $X \sim \mu$ we have $U = \F(\E(X)) \sim \mathcal U([0,1])$. Then, the one-dimensional NeurAM manifold $\gamma$ is defined for $u \in [0,1]$ as
\begin{equation}
\gamma \colon u \mapsto \mathcal D(\mathcal F^{-1}(u)).
\end{equation}
Similarly to the autoencoder and the surrogate, also the CDF $\F$ must be approximated. 
We propose to use the empirical distribution $\widetilde\F$ as represented by the histogram of $\{ \widetilde\E(x_k; \mathfrak e) \}_{k=1}^K$ where $x_k \sim \mu$. 
Notice that $K$ is not limited by the available computational budget, since the cost of evaluating  the encoder is negligible. 
For additional details on the construction of the neural active manifold we refer to~\cite{ZGS25}.

We now show how NeurAM can be used to create a stratification of $\mathbb D$.
Let $\{ a_s \}_{s=0}^{S}$ be a collection of locations in $[0,1]$ such that
\begin{equation}
0 = a_0 < a_1 < \cdots < a_{S-1} < a_S = 1,
\end{equation}
and let $\{ A_s \}_{s=1}^S$ be the partition of $[0,1]$ where $A_s = [a_{s-1}, a_s]$. Then, for all $s = 1, \dots, S$, define
\begin{equation} \label{eq:Ds_def}
D_s = \{ x \in \mathbb D \colon \F(\E(x)) \in A_s \}.
\end{equation}
The main idea underlying this stratification is the reasonable assumption that inputs that are projected to points that are close in the one-dimensional manifold should have similar outputs, resulting in strata with limited variance. 
The next result gives an explicit formula to compute the probability of the stratum $D_s$, and implies that the conditions in equation \eqref{eq:condition_stratification} are satisfied for this particular stratification.

\begin{lemma} \label{lem:mu_lambda}
Let $\{ D_s \}_{s=1}^S$ be defined as in equation \eqref{eq:Ds_def}. Then
\begin{equation}
\mu(D_s) = \lambda(A_s) = a_s - a_{s-1},
\end{equation}
where $\lambda$ denotes the one-dimensional Lebesgue measure.
\end{lemma}
\begin{proof}
By equation \eqref{eq:inverse_transform_sampling} we have
\begin{equation}
\mu(D_s) = \Pr^\mu \left( X \in D_s \right) = \Pr^{\mu} \left( \mathcal F(\mathcal E(X)) \in A_s \right) = \Pr^{\mathcal U([0,1])} \left( U \in A_s \right) = \lambda(A_s) = a_s - a_{s-1},
\end{equation}
which is the desired result.
\end{proof}

Therefore, due to \cref{lem:mu_lambda}, the proposed NeurAM-based stratified estimator and its variance become
\begin{equation} \label{eq:sMC_general_NeurAM}
\begin{aligned}
\widehat q_\sMC &= \sum_{s=1}^S \lambda(A_s) \frac1{N_s} \sum_{n=1}^{N_s} \Q(x_n^{(s)}), \\
\Var[\widehat q_\sMC] &=  \sum_{s=1}^S  \frac{\lambda(A_s)^2}{N_s} \Var^\mu[\Q(X) | X \in D_s],
\end{aligned}
\end{equation}
where the second expression enables direct computation of the estimator variance by approximating the variances within each stratum. The main steps to construct the stratified estimator are summarized in \cref{alg:construction}.

\begin{remark} \label{rem:strata_mu}
The samples $x_n^{(s)}$ in the stratum $D_s$ can be obtained by sampling from the probability distribution $\mu$ and by rejecting the values for which $\F(\E(x_n^{(s)})) \not\in A_s$. By construction, this guarantees that the accepted sample $x_n^{(s)}$ always lies within the domain $\mathbb D$, since all candidates are sampled from $\mu$. Moreover, since acceptance is conditioned on $\F(\E(x_n^{(s)})) \in A_s$, the resulting sample necessarily belongs to $D_s$ by definition. Therefore, it is not possible to produce samples outside the support of $\mathbb D$, or failing to cover portions of $\mathbb D$. In practice, this also means that the method can be used with existing datasets without requiring a specialized sampling procedure.
\end{remark}

\begin{example}[Uniform stratification] \label{ex:uniform}
A simple approach to build a stratification of the unit interval $[0,1]$ is to consider equispaced points $\{ a_s \}_{s=0}^S$, i.e., $a_s = s/S$, which gives $\mu(D_s) = \lambda(A_s) = 1/S$. Using uniform stratification, the estimator and its variance become
\begin{equation}
\begin{aligned}
\widehat q_{\mathrm{sMC}}^{\mathrm u} &= \frac1S \sum_{s=1}^S \frac1{N_s} \sum_{n=1}^{N_s} \Q(x_n^{(s)}), \\
\Var[\widehat q_\sMC^{\mathrm u}] &= \frac1{S^2} \sum_{s=1}^S \frac1{N_s} \Var^\mu[\Q(X) | X \in D_s].
\end{aligned}
\end{equation}
\end{example}

\begin{algorithm}
\caption{NeurAM-based stratified Monte Carlo estimator} \label{alg:construction}
\begin{tabbing}
\textbf{Input:} \= Model $\Q$  \\
\> Input distribution $\mu$ with support $\mathbb D$ or samples from it \\
\> Number of strata $S$ and stratification of the unit interval $\{ a_s \}_{s=0}^S$ \\
\> Computational budget $N$ and allocation strategy $\{ N_s \}_{s=1}^S$
\end{tabbing}
\begin{tabbing}
\textbf{Output:} \= Estimator $\widehat q_\sMC$ of $q = \Ex^\mu[\Q(X)]$
\end{tabbing}
\begin{enumerate}[label=\arabic*:,itemindent=-0.75cm]
\item Construct the NeurAM $(\E, \D, \S)$ by minimizing the loss function $\mathcal L$ in equation \eqref{eq:loss_function}
\item Construct the CDF $\F$ from equation \eqref{eq:CDF}
\item Define the subdomains $\{ A_s \}_{s=1}^S$ as $A_s = [a_{s-1}, a_s]$
\item Compute the coefficients $\{ \lambda(A_s) \}_{s=1}^S$ as $\lambda(A_s) = a_s - a_{s-1}$
\item Define the subdomains $\{ D_s \}_{s=1}^S$ using equation \eqref{eq:Ds_def}
\item Draw the samples $\{ \{ x_n^{(s)} \}_{n=1}^{N_s} \}_{s=1}^S$ such that $x_n^{(s)} \sim \mu |_{D_s}$ as explained in \cref{rem:strata_mu}
\item Compute the estimator $\widehat q_\sMC$ using equation \eqref{eq:sMC_general_NeurAM}
\end{enumerate}
\end{algorithm}

\subsection{Allocation strategies and analysis of the variance} \label{sec:variance_sMC}

A natural question is how to allocate the available budget $N$, i.e., how to choose the samples $\{ N_s \}_{s=1}^S$, in order to get the smallest possible variance in equation \eqref{eq:sMC_general_NeurAM}. The good news is that the same properties of the standard stratified Monte Carlo estimator still hold true. In particular, the optimal allocation that minimizes the variance is given by
\begin{equation}
N_s^{(1)} = \frac{\lambda(A_s)\sqrt{\Var^\mu[\Q(X) | X \in D_s]}}{\sum_{r=1}^S \lambda(A_r)\sqrt{\Var^\mu[\Q(X) | X \in D_r]}}\,N.
\end{equation}
We notice that, to compute $\{ N_s \}_{s=1}^S$, it is necessary to estimate $\Var^\mu[\Q(X) | X \in D_s]$ for all $s = 1, \dots, S$, and this can be done employing the surrogate model $\Q_{\mathrm S}$ provided by NeurAM. 
Nevertheless, a simpler proportional allocation
\begin{equation}
N_s^{(2)} = \lambda(A_s)\,N,
\end{equation}
still guarantees variance reduction with respect to standard Monte Carlo. 
In fact, by replacing these two allocations in the variance~\eqref{eq:sMC_general_NeurAM}, we have
\begin{equation} \label{eq:optimal_variances}
\begin{aligned}
\Var[\widehat q_\sMC^{(1)}] &= \frac1N \left\{ \sum_{s=1}^S \lambda(A_s) \sqrt{\Var^\mu[\Q(X) | X \in D_s]} \right\}^2, \\
\Var[\widehat q_\sMC^{(2)}] &= \frac1N \sum_{s=1}^S \lambda(A_s) \Var^\mu[\Q(X) | X \in D_s],
\end{aligned}
\end{equation}
which, due to Jensen's inequality and the law of total variance, yield
\begin{equation} \label{eq:bound_variances}
\Var[\widehat q_\sMC^{(1)}] \le \Var[\widehat q_\sMC^{(2)}] \le \frac1N \Var^\mu[\Q(X)] = \Var[\widehat q_\MC].
\end{equation}
In particular, we have
\begin{equation}
\Var[\widehat q_\MC] - \Var[\widehat q_\sMC^{(2)}] = \frac1N \sum_{s=1}^S \lambda(A_s) \left( \Ex^\mu[\Q(X) | X \in D_s] - \Ex^\mu[Q(X)] \right)^2,
\end{equation}
which shows that the variance reduction is greater when the differences between the global expectation and the local expectations within each stratum are larger. We also mention that hybrid allocation strategies that use both proportional and optimal allocation have been introduced in \cite{PeK22}.

\begin{figure}
\begin{center}
\includegraphics{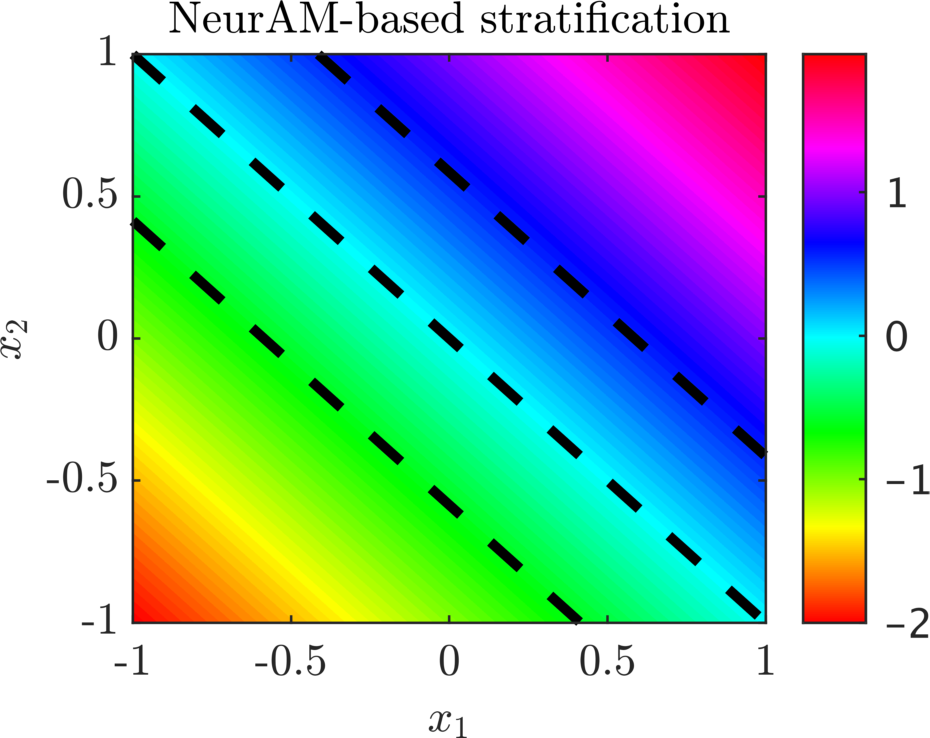} \hspace{1cm}
\includegraphics{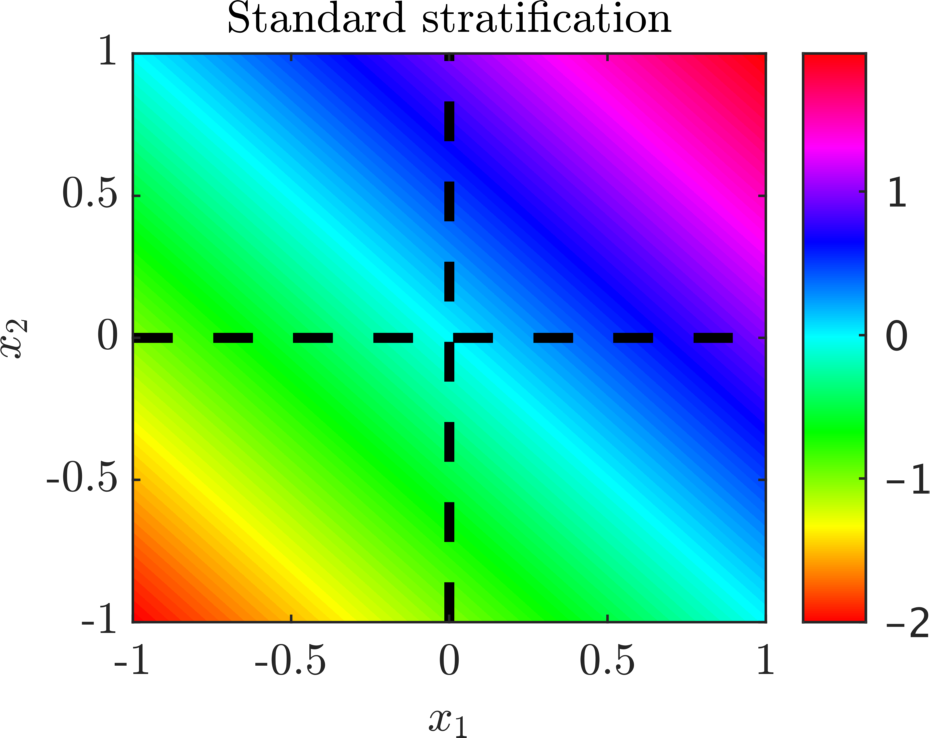}
\end{center}
\caption{Comparison between the NeurAM-based stratification (left) and the standard stratification made with a regular grid (right), for the simple linear model in \cref{ex:example}.}
\label{fig:example}
\end{figure}

\begin{example} \label{ex:example}
To demonstrate the proposed approach and compare it with traditional stratified samping, we consider a linear model for which the NeurAM can be computed analytically. 
Let $\Q(x) = x_1 + x_2$ and $\mu = \mathcal U([-1,1]^2)$. Then, a global minimizer of~\eqref{eq:loss_function} is given by
\begin{equation}
\E(x) = x_1 + x_2, \qquad \D(z) = \begin{bmatrix} \frac{z}2 & \frac{z}2 \end{bmatrix}^\top, \qquad \S(z) = z,
\end{equation}
with NeurAM given by $\{ x \in [-1,1]^2 \colon x_2 = x_1 \}$. Moreover, the latent variables follow a triangular distribution, $\E(X) \sim \mathcal T(-2,0,2)$, with CDF~\eqref{eq:CDF} given by
\begin{equation}
\F(t) = \begin{cases}
0, & \text{ if } t \le -2, \\
\frac18 t^2 + \frac12 t + \frac12, & \text{ if } -2 \le t \le 0, \\
- \frac18 t^2 + \frac12 t + \frac12, & \text{ if } 0 \le t \le 2, \\
1, & \text{ if } t \ge 2.
\end{cases}
\end{equation}
We then choose a number $S = 4$ of strata, and consider a uniform partition of the unit interval
\begin{equation}
A_1 = \left[ 0, \frac14 \right], \qquad A_2 = \left[ \frac14, \frac12 \right], \qquad A_3 = \left[ \frac12, \frac34 \right], \qquad A_4 = \left[ \frac34, 1 \right].
\end{equation}
From equation~\eqref{eq:Ds_def}, the NeurAM-based stratification of $[-1,1]^2$ is given by
\begin{equation}
\begin{aligned}
D_1 &= \left\{ x \in [-1,1]^2 \colon -1 < x_1 < \sqrt2-1, \quad -1 < x_2 < -x_1 + \sqrt2-2 \right\}, \\
D_2 &= \left\{ x \in [-1,1]^2 \colon -1 < x_1 < 1, \quad -x_1 + \sqrt2 - 2 < x_2 < -x_1 \right\}, \\
D_3 &= \left\{ x \in [-1,1]^2 \colon -1 < x_1 < 1, \quad -x_1 < x_2 < -x_1 + 2 - \sqrt2 \right\}, \\
D_4 &= \left\{ x \in [-1,1]^2 \colon 1-\sqrt2 < x_1 < 1, \quad -x_1 + 2 - \sqrt2 < x_2 < 1 \right\},
\end{aligned}
\end{equation}
which is shown in the left plot in \cref{fig:example}. 
We then have $\lambda(A_s) = 1/4$ for all $s = 1, \dots, 4$, and the local variances (we omit their derivation for brevity) are
\begin{equation}
\begin{aligned}
\Var^\mu[\Q(X) | X \in D_1] = \Var^\mu[\Q(X) | X \in D_4] &= \frac19, \\
\Var^\mu[\Q(X) | X \in D_2] = \Var^\mu[\Q(X) | X \in D_3] &= \frac{32}9 \sqrt2 - 5.
\end{aligned}
\end{equation}
Using~\eqref{eq:optimal_variances}, we have
\begin{equation}
\Var[\widehat q_\sMC^{(1)}] = \frac{16\sqrt2 - 22 + \sqrt{32\sqrt2 - 45}}{18N} \qquad \text{and} \qquad \Var[\widehat q_\sMC^{(2)}] = \frac{16\sqrt2 - 22}{9N},
\end{equation}
where $N$ is the computational budget. Let us now compare these results with the variance of a stratified Monte Carlo estimator, where the stratification is given by a regular grid with two subdivisions in each dimension and therefore $S=4$ strata in total
\begin{equation}
\widetilde D_1 = [-1, 0] \times [0, 1], \quad \widetilde D_2 = [0, 1] \times [0, 1], \quad \widetilde D_3 = [-1, 0] \times [-1, 0], \quad \widetilde D_4 = [0, 1] \times [-1, 0],
\end{equation}
which is shown in the right plot in \cref{fig:example}. 
In this case, for all $s = 1, \dots, 4$ we have
\begin{equation}
\Var^\mu[\Q(X) | X \in \widetilde D_i] = \frac16, 
\end{equation}
which implies
\begin{equation}
\Var[\widehat q_\sMC^{(1)}] = \frac1{6N} \qquad \text{and} \qquad \Var[\widehat q_\sMC^{(2)}] = \frac1{6N}.
\end{equation}
Moreover, notice that the variance of standard Monte Carlo is $\Var[\widehat q_\MC] = 2/(3N)$. 
Therefore, independently of the allocation strategy, the variance of the NeurAM-based stratified estimator ($\sim0.06/N$ and $\sim0.07/N$) is significantly smaller than the one given by a regular grid ($\sim0.17/N$), which already improves over the standard Monte Carlo estimator variance ($\sim0.67/N$).
\end{example}

In the following result, we quantify how the accuracy of NeurAM, i.e., the magnitude of the achieved minimum of the loss function $\mathcal L$ in equation \eqref{eq:loss_function}, influences the variance of the stratified Monte Carlo estimator.

\begin{theorem} \label{thm:bound_var}
Let $\E, \D, \S$ satisfy $\mathcal L(\E, \D, \S) = \epl \ge 0$, where $\mathcal L$ is defined in equation \eqref{eq:loss_function}, and assume $\F$ to be invertible. Then, it holds for all $s = 1, \dots, S$
\begin{equation}
\Var^\mu[\Q(X) | X \in D_s] \le 2 \Var^{\mathcal U}[\S(\F^{-1}(U)) | U \in A_s] + \frac{2\epl}{\lambda(A_s)}.
\end{equation}
Moreover, we have
\begin{equation}
\Var[\widehat q_\sMC] \le 2 \sum_{s=1}^S \frac{\lambda(A_s)^2}{N_s} \Var^\mu[\S(\F^{-1}(U)) | U \in A_s] + 2 \epl \sum_{s=1}^S \frac{\lambda(A_s)}{N_s}.
\end{equation}
\end{theorem}
\begin{proof}
By Cauchy--Schwarz inequality and properties of conditional expectations, we have
\begin{equation}
\begin{aligned}
\Var^\mu[\Q(X) | X \in D_s] &\le 2 \left( \Var^\mu[\S(\E(X)) | X \in D_s] + \Var^\mu[\Q(X) - \S(\E(X)) | X \in D_s] \right) \\
&\le 2 \left( \Var^\mu[\S(\E(X)) | X \in D_s] + \Ex^\mu[(\Q(X) - \S(\E(X)))^2 | X \in D_s] \right) \\
&\le 2 \left( \Var^\mu[\S(\E(X)) | X \in D_s] + \frac{\Ex^\mu[(\Q(X) - \S(\E(X)))^2]}{\mu(D_s)} \right).
\end{aligned}
\end{equation}
Then, notice that by definition of $D_s$ and \cref{lem:mu_lambda}, we get $\mu(D_s) = \lambda(A_s)$ and
\begin{equation}
\begin{aligned}
\Var^\mu[\S(\E(X)) | X \in D_s] &= \Var^\mu[\S(\F^{-1}(\F(\E(X)))) | \F(\E(X)) \in A_s] \\
&= \Var^{\mathcal U}[\S(\F^{-1}(U)) | U \in A_s],
\end{aligned}
\end{equation}
which, together with the fact that 
\begin{equation}
\Ex^\mu[(\Q(X) - \S(\E(X)))^2] \le \mathcal L(\E, \D, \S) = \epl,
\end{equation}
implies the first desired result. The second estimate follows directly from the first by equation \eqref{eq:sMC_general_NeurAM}.
\end{proof}

Applying \cref{thm:bound_var}, we can obtain upper bounds for the variances of the stratified Monte Carlo estimators $\widehat q_\sMC^{(1)}$ and $\widehat q_\sMC^{(2)}$ with optimal and proportional allocation strategies, respectively. Specifically, using equation \eqref{eq:optimal_variances}, we deduce
\begin{equation} \label{eq:optimal_variances_bounds}
\begin{aligned}
\Var[\widehat q_\sMC^{(1)}] &\le \frac2N \left\{ \sum_{s=1}^S \lambda(A_s) \sqrt{\Var^{\mathcal U}[\S(\F^{-1}(U)) | U \in A_s]} \right\}^2 + \frac{2S\epl}N, \\
\Var[\widehat q_\sMC^{(2)}] &\le \frac2N \sum_{s=1}^S \lambda(A_s) \Var^{\mathcal U}[[\S(\F^{-1}(U)) | U \in A_s] + \frac{2S\epl}N.
\end{aligned}
\end{equation}
In \cref{thm:bound_var}, and consequently in equation \eqref{eq:optimal_variances_bounds}, we see that an additional term appears that depends on the accuracy of the NeurAM. On the other hand, the first term reduces, up to the constant factor $2$, to the variance of the stratified estimator for the one-dimensional model $\mathcal{S} \circ \mathcal{F}^{-1}$ with respect to the uniform distribution $\mathcal{U}([0,1])$. This allows us to apply existing theoretical results for one-dimensional stratification. In particular, we obtain the following result.
\begin{corollary} \label{cor:uniform_var_bound}
Let $\E, \D, \S$ satisfy $\mathcal L(\E, \D, \S) = \epl \ge 0$, where $\mathcal L$ is defined in equation \eqref{eq:loss_function}, and assume $\F$ to be invertible and $\S \circ \F^{-1}$ to be Lipschitz continuous. Using the uniform stratification of \cref{ex:uniform}, it holds
\begin{equation}
\Var[\widehat q_\sMC^{\mathrm u, (i)}] \le \frac{\norm{(\S \circ \F^{-1})'}^2_{L^\infty([0,1])}}{6 S^2 N} + \frac{2S\epl}N,
\end{equation}
for both $i=1,2$, i.e., optimal and proportional allocations, respectively.
\end{corollary}
\begin{proof}
Following \cite[Chapter V, Example 7.1]{AsG07}, we have
\begin{equation}
\Var^{\mathcal U}[[\S(\F^{-1}(U)) | U \in A_s] \le \frac{\norm{(\S \circ \F^{-1})'}^2_{L^\infty([0,1])}}{12 S^2},
\end{equation}
which, together with equation \eqref{eq:optimal_variances_bounds}, yields the desired result.
\end{proof}

\begin{remark}
\cref{cor:uniform_var_bound} shows that the variance scales as $\mathcal O(S^{-2})$, but only up to the point where the error introduced by the NeurAM, quantified by $\epl$, becomes dominant. Beyond this point, further refining the stratification may not reduce the variance and can even cause it to increase. By \cref{cor:uniform_var_bound}, the optimal number of strata is
\begin{equation}
S^* = \min \left\{ \left\lceil \left( \frac{\norm{(\S \circ \F^{-1})'}^2_{L^\infty([0,1])}}{6\epl} \right)^{1/3} \right\rceil, \; N \right\},
\end{equation}
which highlights that increasing the stratification is beneficial only when the NeurAM approximation is sufficiently accurate. Moreover, if we want the variance to vanish as $N \to \infty$, then even when $\epl \ll 1$, we cannot choose $S=N$, since the error term would remain of order $\mathcal O(\epl)$. We finally remark that this analysis is independent of the dimension $d$ of the domain $\mathbb D$ of the model $\Q$, which is one of the main advantages of this approach.
\end{remark}

After investigating optimal sample allocations for a fixed collection of strata, in the next section we consider the problem of determining an optimal stratification $\{ A_s \}_{s=1}^S$ that leads to reduced variances in~\eqref{eq:optimal_variances}. 

\subsection{A heuristic algorithm to improve the stratification} \label{sec:heuristic}

In this section we focus on determining an \emph{optimal} stratification leading to estimators with minimal variance. 
Considering the two allocation strategies discussed in the previous section, and due to the bound \eqref{eq:bound_variances}, this amounts to solving the minimization problems
\begin{equation} \label{eq:optimization_strata_M}
\begin{aligned}
(\mathrm M1) \quad& \min_{0 = a_0 < a_1 < \cdots < a_{S-1} < a_S = 1} \sum_{s=1}^S \lambda(A_s) \sqrt{\Var^\mu[\Q(X) | X \in D_s]}, \\
(\mathrm M2) \quad& \min_{0 = a_0 < a_1 < \cdots < a_{S-1} < a_S = 1} \sum_{s=1}^S \lambda(A_s) \Var^\mu[\Q(X) | X \in D_s]. \\
\end{aligned}
\end{equation}
However, (M1) and (M2) might be challenging and computationally expensive to solve, so that it might be more effective to simply use a uniform stratification as discussed in \cref{ex:uniform}. 
In addition, we would need to solve nonlinear optimization problems in high dimensions, where the variances in the objective functions are not even known explicitly, but must be approximated. 
Therefore, we propose a heuristic strategy that iteratively refines the strata, while keeping the computational cost limited. 

We proceed as follows. Let $\mathcal A^{(1)}, \Lambda^{(1)}, \Sigma^{(1)}$ be defined as
\begin{equation}
\begin{aligned}
\mathcal A^{(1)} &= \{ a_0^{(1)}, a_1^{(1)} \} = \{ 0, 1 \}, \\
\Lambda^{(1)} &= \{ \lambda_1^{(1)} \}, \quad \lambda_1^{(1)} = a_1^{(1)} - a_0^{(1)} = 1, \\
\Sigma^{(1)} &= \{ \sigma_1^{(1)} \}, \quad \sigma_1^{(1)} = \Var^\mu \left[ \Q(X) | \F(\E(X)) \in [a_0^{(1)}, a_1^{(1)}] \right] = \Var^\mu[\Q(X)].
\end{aligned}
\end{equation}
Then, at each iteration we bisect the interval that contributes the most to the overall variance.
In particular, for all $j = 2, \dots, S$, select the index $i^*$ such that
\begin{equation}
i^* = 
\begin{cases}
\argmax_{i \in \{ 1, \dots, j-1 \}} \lambda_i^{(j-1)} \sqrt{\sigma_i^{(j-1)}}, & \text{ for allocation (1)}, \\ 
\argmax_{i \in \{ 1, \dots, j-1 \}} \lambda_i^{(j-1)} \sigma_i^{(j-1)}, & \text{ for allocation (2)}, \\ 
\end{cases}
\end{equation}
and pick a point $a^* \in [a_{i^*-1}^{(j-1)}, a_{i^*}^{(j-1)}]$. Then, define $\mathcal A^{(j)}, \Lambda^{(j)}, \Sigma^{(j)}$ as
\begin{equation}
\begin{aligned}
\mathcal A^{(j)} &= \{ a_i^{(j)} \}_{i=0}^j = \mathcal A^{(j-1)} \cup \{ a^* \}, \\
\Lambda^{(j)} &= \{ \lambda_i^{(j)} \}_{i=0}^j, \quad \lambda_i^{(j)} = a_i^{(j)} - a_{i-1}^{(j)}, \\
\Sigma^{(j)} &= \{ \sigma_i^{(j)} \}_{i=0}^j, \quad \sigma_i^{(j)} = \Var^\mu \left[ \Q(X) | \F(\E(X)) \in [a_{i-1}^{(j)}, a_i^{(j)}] \right].
\end{aligned}
\end{equation}
The resulting stratification is given by the points in $\mathcal A^{(S)}$ after the last iteration $j = S$. 
Finally, we need to specify how to select $a^*$ at each step. 
We emphasize that, due to the law of total variance, any point $a^*$ can be chosen, meaning that this procedure does not increase the variance of the resulting estimator independently of the selected point $a^*$. 
Nevertheless, we list here two alternatives that can be used in practice:
\begin{itemize}[leftmargin=*]
\item interval mid-point, i.e., $a^*_{\mathrm h} = (a_{i^*-1}^{(j-1)} + a_{i^*}^{(j-1)})/2$;
\item optimal value given by
\begin{equation}
a^*_{\mathrm o} = 
\begin{cases}
\begin{aligned}
\argmin_{a^* \in [a_{i^*-1}^{(j-1)}, a_{i^*}^{(j-1)}]} &\left\{ (a^* - a_{i^*-1}^{(j-1)}) \sqrt{\Var^\mu \left[ \Q(X) | \F(\E(X)) \in [a_{i^*-1}^{(j-1)}, a^*] \right]} \right. \\
&\left. + (a_{i^*}^{(j-1)} - a^*) \sqrt{\Var^\mu \left[ \Q(X) | \F(\E(X)) \in [a^*, a_{i^*}^{(j-1)}] \right]} \right\},
\end{aligned}
& \text{ for (1)}, \\
\begin{aligned}
\argmin_{a^* \in [a_{i^*-1}^{(j-1)}, a_{i^*}^{(j-1)}]} &\left\{ (a^* - a_{i^*-1}^{(j-1)}) \Var^\mu \left[ \Q(X) | \F(\E(X)) \in [a_{i^*-1}^{(j-1)}, a^*] \right] \right. \\
&\left. + (a_{i^*}^{(j-1)} - a^*) \Var^\mu \left[ \Q(X) | \F(\E(X)) \in [a^*, a_{i^*}^{(j-1)}] \right] \right\},
\end{aligned}
& \text{ for (2)}. 
\end{cases}
\end{equation}
\end{itemize}
Note that, in the latter case, the minimum is well-defined since the function to be minimized is continuous on a compact interval. 
Moreover, this choice provides the best variance reduction at each step, but it requires additional computations (one needs to solve a minimization problem), even if the surrogate model $\Q_{\mathrm S}$ can be use to estimate the variance of the strata.
Also, the optimization problem is, in this case, only one-dimensional, unlike (M1) and (M2) in \eqref{eq:optimization_strata_M}.

\section{Application to multifidelity estimators} \label{sec:multifidelity} 

The NeurAM-based stratification presented here can be integrated with other variance reduction methods to achieve even greater precision in the resulting estimates.
For example, it can be combined with multifidelity Monte Carlo estimators, as explained next. 
Let $\Q^\HF = \Q$ and assume that a cheap low-fidelity approximation $\Q^\LF$ of the high-fidelity model $\Q^\HF$ is also known. 
Let $w = \mathcal C^\LF / \mathcal C^\HF$ be the cost ratio between the two model fidelities. 
The available computational budget $N$, expressed in terms of high-fidelity model evaluations, can then be assembled from the high- and low-fidelity model evaluations $N^\HF$ and $N^\LF$, respectively, as
\begin{equation}
N = N^\HF + w N^\LF.
\end{equation}
It is possible to determine $N^\HF$ and $N^\LF$ so the resulting multifidelity estimator has minimum variance, or, in other words, an \emph{optimal allocation} can be computed in closed-form as~\cite{PWG16}
\begin{equation} \label{eq:optimal_allocation}
N^\HF = \frac{N}{1 + w \beta} \quad \text{and} \quad N^\LF = \beta N^\HF = \frac{\beta N}{1 + w \beta}, \quad \text{with} \quad \beta = \sqrt{\frac{\rho^2}{w(1 - \rho^2)}},
\end{equation}
where $\rho$ is the Pearson correlation coefficient between the two fidelities
\begin{equation} \label{eq:correlation}
\rho = \frac{\Cov^\mu \left[ \Q^\HF(X), \Q^\LF(X) \right]}{\sqrt{\Var^\mu[\Q^\HF(X)] \Var^\mu \left [\Q^\LF(X) \right]}}.
\end{equation}
The multifidelity Monte Carlo estimator is then defined as
\begin{equation} \label{eq:MFMC}
\widehat q_\MFMC = \frac1{N^\HF} \sum_{n=1}^{N^\HF} \Q^\HF(x_n) - \alpha \left( \frac1{N^\HF} \sum_{n=1}^{N^\HF} \Q^\LF(x_n) - \frac1{N^\LF} \sum_{n=1}^{N^\LF} \Q^\LF(x_n) \right),
\end{equation}
where the optimal value for the coefficient $\alpha$ is
\begin{equation} \label{eq:optimal_coefficient}
\alpha = \frac{\Cov^\mu \left[ \Q^\HF(X), \Q^\LF(X) \right]}{\Var^\mu \left[ \Q^\LF(X) \right]},
\end{equation}
and the samples $\{ x_n \}_{n=1}^{N^\LF}$ with $N^\LF > N^\HF$ are drawn from the distribution $\mu$. 
We remark that the estimator $\widehat q_\MFMC$ is unbiased, i.e., $\Ex[\widehat q_\MFMC] = q$, and has variance
\begin{equation} \label{eq:variance_MFMC}
\Var \left[ \widehat q_\MFMC \right] = \frac1N \Var^\mu \left[ \Q^\HF(X) \right] \left( \sqrt{1 - \rho^2} + \sqrt{w\rho^2} \right)^2 = \Var \left[ \widehat q_\MC \right] \left( \sqrt{1 - \rho^2} + \sqrt{w\rho^2} \right)^2.
\end{equation}
This implies that $\Var \left[ \widehat q_\MFMC \right] < \Var \left[ \widehat q_\MC \right]$, leading to variance reduction under the condition $\rho^2 > 4w/(1 + w)^2$. 

A multifidelity approach can be combined with stratified sampling by replacing the Monte Carlo estimator with the multifidelity Monte Carlo estimator in each stratum. In particular, for a NeurAM based stratification $\{ (A_s,D_s) \}_{s=1}^S$ and a collection of samples $\{ \{ x_n^{(s)} \}_{n=1}^{N^\LF_s} \}_{s=1}^S$ with $x_n^{(s)} \sim \mu |_{D_s}$, we have
\begin{equation} \label{eq:estimator_sMFMC_def}
\begin{aligned}
\widehat q_\sMFMC &= \sum_{s=1}^S \lambda(A_s) \left[ \frac1{N^\HF_s} \sum_{n=1}^{N^\HF_s} \Q^\HF(x_n^{(s)}) \right. \\
&\hspace{2.5cm} \left. - \alpha_s \left( \frac1{N^\HF_s} \sum_{n=1}^{N^\HF_s} \Q^\LF(x_n^{(s)}) - \frac1{N^\LF_s} \sum_{n=1}^{N^\LF_s} \Q^\LF(x_n^{(s)}) \right) \right],
\end{aligned}
\end{equation}
where, given the available computational budget in each stratum $\{ N_s \}_{s=1}^S$ such that $N = \sum_{s=1}^S N_s$ and still following \cite{PWG16}, an optimal allocation is given by
\begin{equation}
N^\HF_s = \frac{N_s}{1 + w \beta_s} \quad \text{and} \quad N^\LF = \beta_s N^\HF_s = \frac{\beta_s N_s}{1 + w \beta_s}, \quad \text{with} \quad \beta_s = \sqrt{\frac{\rho_s^2}{w(1 - \rho_s^2)}},
\end{equation}
where $\rho_s$ is the Pearson correlation coefficient in each stratum
\begin{equation}
\rho_s = \frac{\Cov^\mu \left[ \Q^\HF(X), \Q^\LF(X) | X \in D_s \right]}{\sqrt{\Var^\mu[\Q^\HF(X) | X \in D_s] \Var^\mu [\Q^\LF(X)| X \in D_s]}}.
\end{equation}
Moreover, the optimal values for the coefficients $\{ \alpha_s \}_{s=1}^S$ are
\begin{equation}
\alpha_s = \frac{\Cov^\mu \left[ \Q^\HF(X), \Q^\LF(X) | X \in D_s \right]}{\Var^\mu \left[ \Q^\LF(X) | X \in D_s \right]}.
\end{equation}
We notice that the stratified multifidelity Monte Carlo estimator remains unbiased, i.e., it holds $\Ex[\widehat q_\sMFMC] = q$, and, due to equations \eqref{eq:sMC_general_NeurAM} and \eqref{eq:variance_MFMC}, its variance is
\begin{equation} \label{eq:variance_sMFMCM}
\Var[\widehat q_\sMFMC] = \sum_{s=1}^S \frac{\lambda(A_s)^2}{N_s} \Var^\mu[\Q^\HF(X) | X \in D_s] \left( \sqrt{1 - \rho_s^2} + \sqrt{w\rho_s^2} \right)^2,
\end{equation}
which is dependent on the local correlations $\{ \rho_s \}_{s=1}^S$ in each stratum. In the next section we study how the variance can be optimized and provide conditions under which this approach is more effective than standard multifidelity Monte Carlo.

\begin{remark}
NeurAM is applicable to both low- and high-fidelity models to establish a shared space~\cite{ZGS25}. This space can be used to reparameterize the low-fidelity model, thereby enhancing the correlation $\rho$ between the fidelities. For the remainder of this discussion, we assume that $\Q^\LF$ has already undergone reparameterization and demonstrates a strong correlation with $\Q^\HF$, satisfying the condition $\rho^2 > 4w/(1 + w)^2$. 
Although this reparameterization would provide an additional reduction in variance, it is not needed for the discussion that follows.
\end{remark}

\subsection{Allocation strategies and analysis of the variance}

Similar to the discussion in \cref{sec:variance_sMC}, we analyze two allocation strategies for the stratified multifidelity Monte Carlo estimator. 
An allocation that yields estimators with minimum variance involves selecting $N_s$ samples in stratum $s$ according to
\begin{equation} \label{eq:allocation1_MF}
N_s^{(1)} = \frac{\lambda(A_s)\sqrt{\Var^\mu[\Q(X) | X \in D_s]} \left( \sqrt{1 - \rho_s^2} + \sqrt{w\rho_s^2} \right)}{\sum_{r=1}^S \lambda(A_r)\sqrt{\Var^\mu[\Q(X) | X \in D_r]} \left( \sqrt{1 - \rho_r^2} + \sqrt{w\rho_r^2} \right)} N,
\end{equation}
while a proportional allocation  is still
\begin{equation} \label{eq:allocation2_MF}
N_s^{(2)} = \lambda(A_s)\,N.
\end{equation}
Replacing the two allocation strategies in equation \eqref{eq:variance_sMFMCM} gives
\begin{equation}
\begin{aligned}
\Var[\widehat q_\sMFMC^{(1)}] &= \frac1N \left\{ \sum_{s=1}^S \lambda(A_s) \sqrt{\Var^\mu[\Q(X) | X \in D_s]} \left( \sqrt{1 - \rho_s^2} + \sqrt{w\rho_s^2} \right) \right\}^2, \\
\Var[\widehat q_\sMFMC^{(2)}] &= \frac1N \sum_{s=1}^S \lambda(A_s) \Var^\mu[\Q(X) | X \in D_s] \left( \sqrt{1 - \rho_s^2} + \sqrt{w\rho_s^2} \right)^2,
\end{aligned}
\end{equation}
and, from Jensen's inequality, it is also easy to see that
\begin{equation}
\Var[\widehat q_\sMFMC^{(1)}] \le \Var[\widehat q_\sMFMC^{(2)}].
\end{equation}
It is evident and straightforward to verify that, if all the local correlations satisfy the condition that ensures variance reduction, i.e., $\rho_s^2 > 4w/(1+w)^2$ for all $s = 1, \dots, S$, then
\begin{equation}
\Var[\widehat q_\sMFMC^{(1)}] \le \Var[\widehat q_\sMC^{(1)}] \quad \text{and} \quad 
\Var[\widehat q_\sMFMC^{(2)}] \le \Var[\widehat q_\sMC^{(2)}].
\end{equation}
Moreover, the following result provides conditions under which stratification leads to multifidelity Monte Carlo estimator with reduced variance.

\begin{proposition} \label{pro:variance_MF}
Let $\widehat q_\sMFMC^{(1)}$ and $\widehat q_\sMFMC^{(2)}$ be the estimators defined in equation \eqref{eq:estimator_sMFMC_def} with computational budget $N$ and allocations given by \eqref{eq:allocation1_MF} and \eqref{eq:allocation2_MF}, respectively, and let $\rho$ and $\{ \rho_s \}_{s=1}^S$ be the overall and intra-stratum correlations between the high- and low-fidelity model. 
Assuming $\rho^2 \ge 4w/(1+w)^2$, the following two statements hold
\begin{enumerate}[leftmargin=*]
\item if $\bar\rho \ge \rho$, where $\bar\rho = \sum_{s=1}^S \lambda(A_s)\,\rho_s$, then $\Var[\widehat q_\sMFMC^{(1)}] \le \Var[\widehat q_\MFMC]$.
\item if $\rho_s \ge \rho$ for all $s = 1, \dots, S$, then $\Var[\widehat q_\sMFMC^{(2)}] \le \Var[\widehat q_\MFMC]$,
\end{enumerate}
\end{proposition}
\begin{proof}
First, statement $(ii)$ follows from the law of total variance and the fact that the function $f(x) = (\sqrt{1 - x^2} + \sqrt{wx^2})^2$ is decreasing if $x^2 \ge 4w/(1+w)^2$. Then, applying Cauchy--Schwarz inequality, we have
\begin{equation}
\Var[\widehat q_\sMFMC^{(1)}] \le \frac1N \left( \sum_{s=1}^S \lambda(A_s) \Var^\mu[\Q(X) | X \in D_s] \right) \left( \sum_{s=1}^S \lambda(A_s) \left( \sqrt{1 - \rho_s^2} + \sqrt{w\rho_s^2} \right)^2 \right),
\end{equation}
which, due to the concavity of the function $f$ and from Jensen's inequality, leads to
\begin{equation}
\begin{aligned}
\Var[\widehat q_\sMFMC^{(1)}] &\le \frac1N \left( \sum_{s=1}^S \lambda(A_s) \Var^\mu[\Q(X) | X \in D_s] \right) \\
&\hspace{2cm}\times \left( \sqrt{1 - \left( \sum_{s=1}^S \lambda(A_s) \rho_s \right)^2} + \sqrt{w \left( \sum_{s=1}^S \lambda(A_s) \rho_s \right)^2} \right)^2 \\
&= \frac1N \left( \sum_{s=1}^S \lambda(A_s) \Var^\mu[\Q(X) | X \in D_s] \right) \left( \sqrt{1 - \bar\rho^2} + \sqrt{w\bar\rho^2} \right)^2.
\end{aligned}
\end{equation}
Using again the law of total variance and the fact that the function $f$ is decreasing gives statement $(i)$, which completes the proof.
\end{proof}

\begin{remark} \label{rem:variance_MF}
In \cref{pro:variance_MF}, the condition for allocation (1) is weaker than the condition for allocation (2), but allocation (2) is easier to apply, since it only requires the probability mass of each strata. Moreover, we notice that \cref{pro:variance_MF} only provides sufficient conditions for variance reduction. In particular, a smaller variance can also be obtained even if the assumptions are not satisfied.
\end{remark}
The heuristic algorithm proposed in \cref{sec:heuristic} can still be applied in the context of stratified multifidelity estimators.
However, at each iteration, one must remember to always multiply the variance $\widetilde\sigma$ of the high-fidelity model in each stratum by a coefficient $\widetilde\eta$ given by
\begin{equation}
\widetilde\eta = \left( \sqrt{1 - \widetilde\rho^2} + \sqrt{w \widetilde\rho^2} \right)^2,
\end{equation}
which depends on the correlation $\widetilde\rho$ in the stratum and quantifies the variance reduction.

\section{Numerical experiments}\label{sec:experiments}

\begin{figure}
\begin{center}
\includegraphics{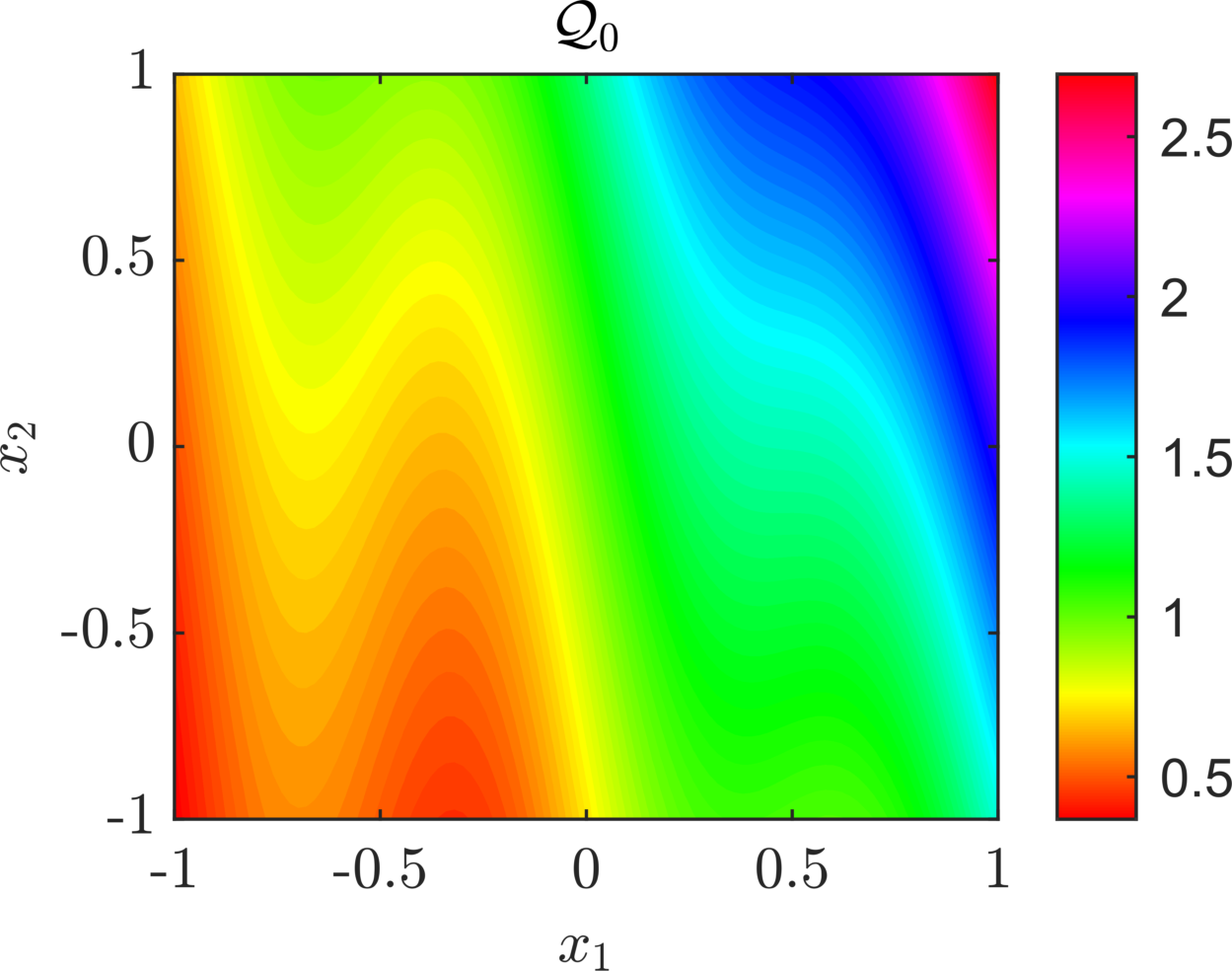}
\end{center}
\caption{Contour plot of the model $\Q_0$ which is used as a test case for the numerical experiments in \cref{sec:num_param,sec:num_heuristic,sec:num_AS,sec:num_multifidelity}.}
\label{fig:model}
\end{figure}

In this section we analyze the performance and the properties of NeurAM-based stratified estimators. 
We first study in \cref{sec:num_param} how their convergence depends on the quantities $M,K$ and on the number of strata $S$. 
We remind the reader that $M$ represents the size of the training dataset used to evaluate the loss function \eqref{eq:loss_NeurAM_approx}, while $K$ is the number of latent space samples used to approximate the CDF in equation \eqref{eq:CDF}.
Next, in \cref{sec:num_heuristic}, we show that the heuristic algorithm introduced in \cref{sec:heuristic} leads to a lower variance with respect to the uniform stratification of the unit interval.
In \cref{sec:num_AS}, we compare the proposed stratification strategy with a similar approach, where NeurAM is replaced by active subspace (AS) -- a linear dimensionality reduction technique \cite{CDW14}.
Then, in \cref{sec:num_multifidelity}, we verify that stratified sampling can be successfully combined with multifidelity estimators for additional variance reduction. 

In all the numerical experiments in \cref{sec:num_param,sec:num_heuristic,sec:num_AS,sec:num_multifidelity} we consider the domain $\mathbb D = [-1,1]^2$, the distribution of the input parameters $\mu_0 = \mathcal U([-1,1]^2)$, and the two-dimensional function $\Q_0 \colon \mathbb D \to \R$ 
\begin{equation} \label{eq:Q0}
\Q_0(x) = e^{0.7x_1 + 0.3x_2} + 0.15\sin(2\pi x_1),
\end{equation}
which has already been used in the context of sampling estimators for uncertainty propagation~\cite{GEG18,ZGS24a,ZGS24b,ZGS25}. 
Its contour plot is shown in \cref{fig:model}, and the quantity of interest $q$ can be computed explicitly as
\begin{equation}
q = \Ex^{\mu_0}[\Q_0(X)] = \frac{25}{21} \left( e^{-1} - e^{-\frac25} - e^{\frac25} + e\right).
\end{equation}

Finally, in \cref{sec:num_highDimensional,sec:num_other,sec:Darcy} we consider higher-dimensional models, we compare our approach to other sampling-based variance reduction strategies, and study a classical UQ test problem, a Darcy-flow problem, involving partial differential equations, respectively. In particular, we show that our methodology, differently from standard stratified sampling, has the potential to maintains its desirable properties in high dimensions.
The code used to perform the following numerical experiments is available at the GitHub repository \url{https://github.com/AndreaZanoni/StratifiedSamplingNeurAM}.

\begin{remark}
The setup is consistent for all numerical experiments in this section.
Specifically, the NeurAM encoder $\widetilde\E(\cdot; \mathfrak e)$, decoder $\widetilde\D(\cdot; \mathfrak d)$, and surrogate model $\widetilde\S(\cdot; \mathfrak s)$ are each parameterized as fully connected neural networks with $2$ hidden layers of $8$ neurons each, and $\operatorname{tanh}$ activation.
Training is performed by minimizing \eqref{eq:loss_NeurAM_approx} using the Adam optimizer with a fixed learning rate of $10^{-3}$ over 10,000 epochs. 
Moreover, in each experiment, we compute 1,000 realizations of the stratified Monte Carlo estimator, and use the empirical mean and standard deviation from these realizations to draw a Gaussian approximation in the plots.
The computational budget is always denoted by $N$ and does not include the $M$ samples employed to train NeurAM, which could, however, be reused in computationally expensive real applications. 
\end{remark}

\subsection{Sensitivity to $M$, $K$ and $S$} \label{sec:num_param}

\begin{figure}
\begin{center}
\includegraphics{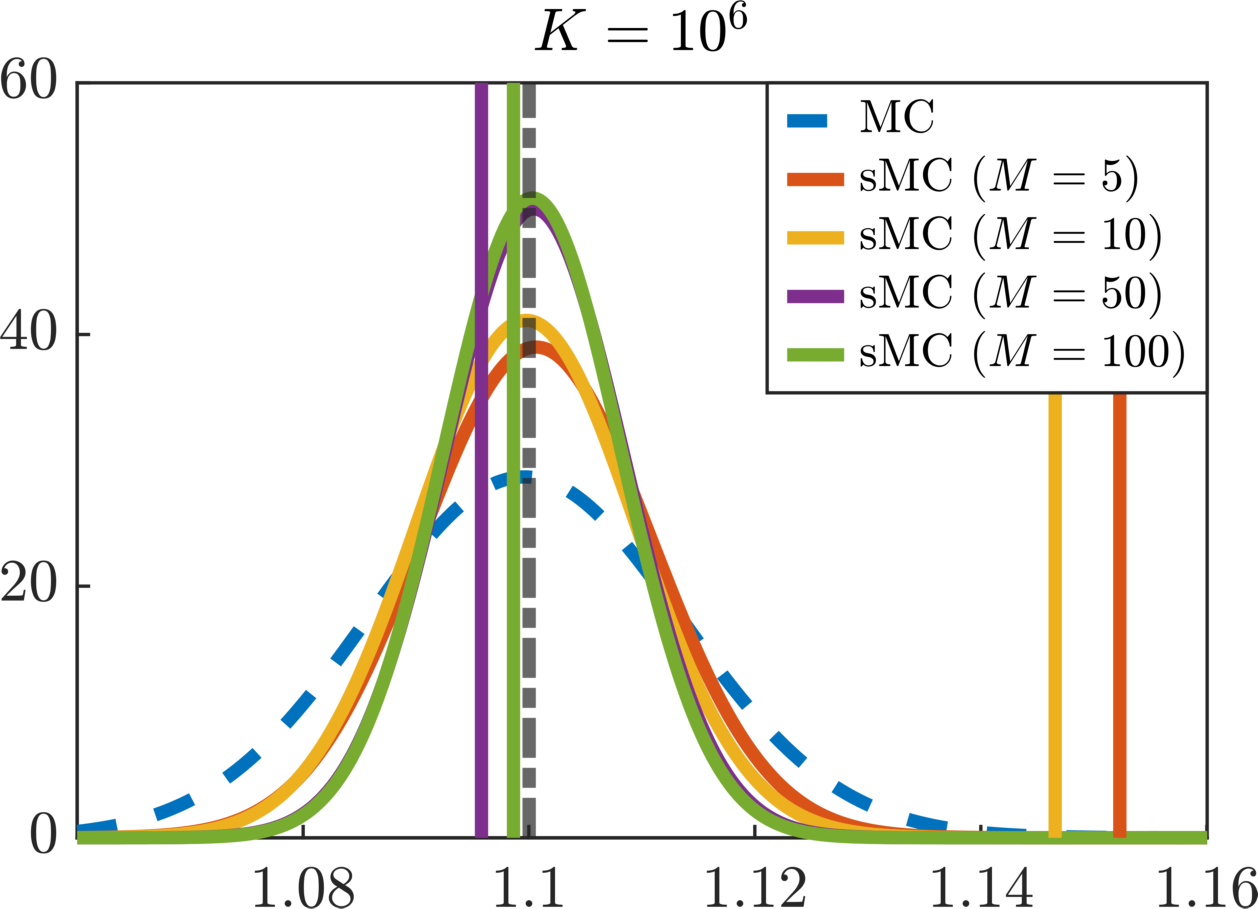} \hspace{1cm}
\includegraphics{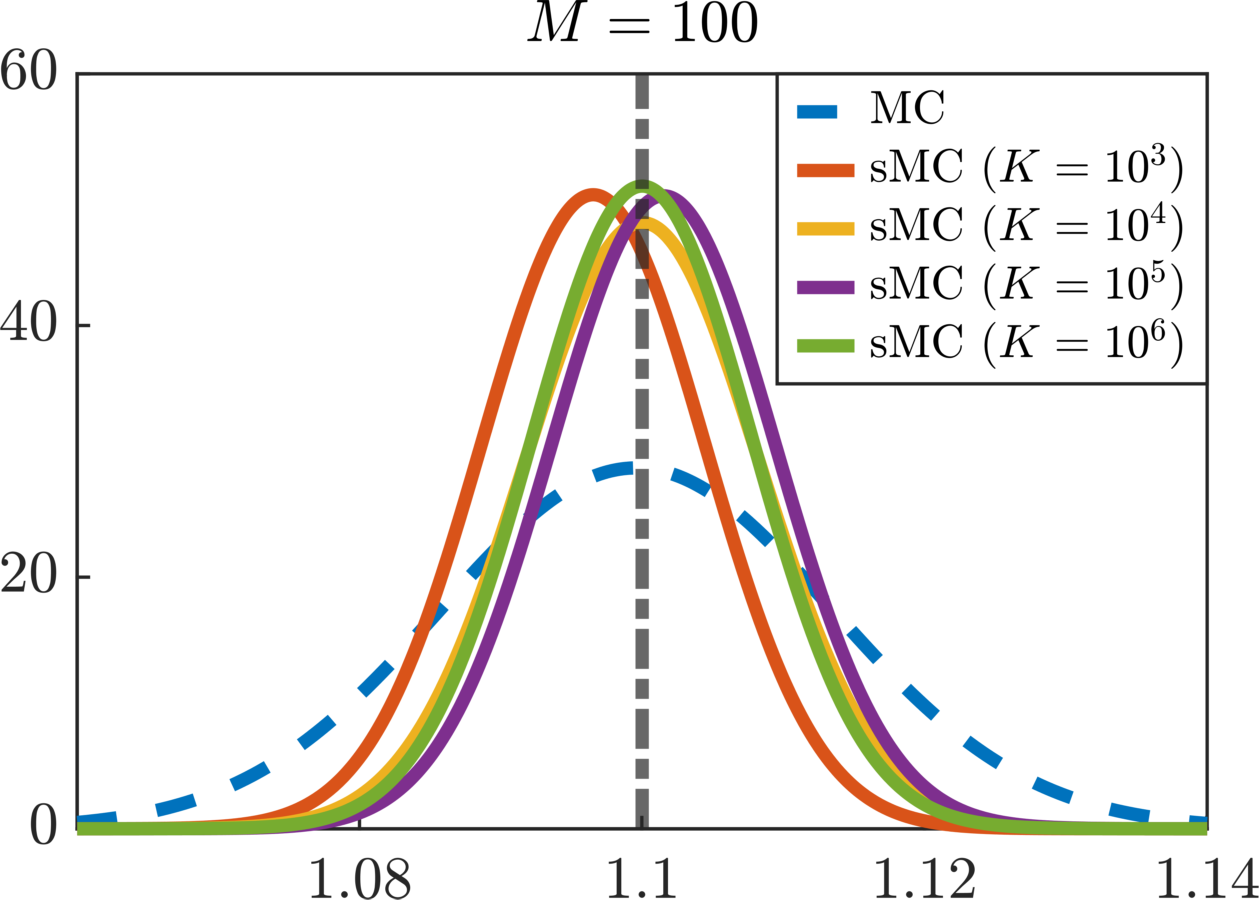}
\end{center}
\vspace{0.25cm}
\begin{center}
\begin{tabular}{c|cccc}
\toprule
\diagbox{$M$}{$K$} & $K = 10^3$ & $K = 10^4$ & $K = 10^5$ & $K = 10^6$ \\
\midrule
$M = 5$ & 1.29e-4 (6.68e-01) & 1.00e-4 (5.18e-01) & 9.63e-5 (4.99e-01) & 1.05e-4 (5.44e-01) \\
$M = 10$ & 6.71e-5 (3.48e-01) & 7.01e-5 (3.63e-01) & 7.07e-5 (3.66e-01) & 9.41e-5 (4.88e-01) \\
$M = 50$ & 7.96e-5 (4.12e-01) & 6.99e-5 (3.62e-01) & 6.53e-5 (3.38e-01) & 6.38e-5 (3.31e-01) \\
$M = 100$ & 7.47e-5 (3.87e-01) & 6.85e-5 (3.55e-01) & 6.51e-5 (3.37e-01) & 6.09e-5 (3.16e-01) \\
\bottomrule
\end{tabular}
\end{center}
\caption{\emph{TOP}: Comparison between standard Monte Carlo estimator $\widehat q_\MC$ (dashed line) and the proposed estimator $\widehat q_\sMC$ (solid line), varying the number of data $M = 5, 10, 50, 100$ (left), and $K = 10^3, 10^4, 10^5, 10^6$ (right), used to learn the NeurAM and the CDF, respectively. The the gray dash-dotted vertical line represents the exact value of the quantity of interest, while the colored solid vertical lines (left) represent the bias obtained using only the surrogate model $\Q_{\mathrm S}$ for the different values of $M$. \emph{BOTTOM}: Mean squared error (MSE) of the estimators to be compared with the value 1.93e-4 for Monte Carlo obtained with $N=1024$. The numbers in parentheses denote the ratio between the MSE of each estimator and the Monte Carlo reference value.}
\label{fig:dependence_KM}
\end{figure}

\begin{figure}
\begin{center}
\includegraphics{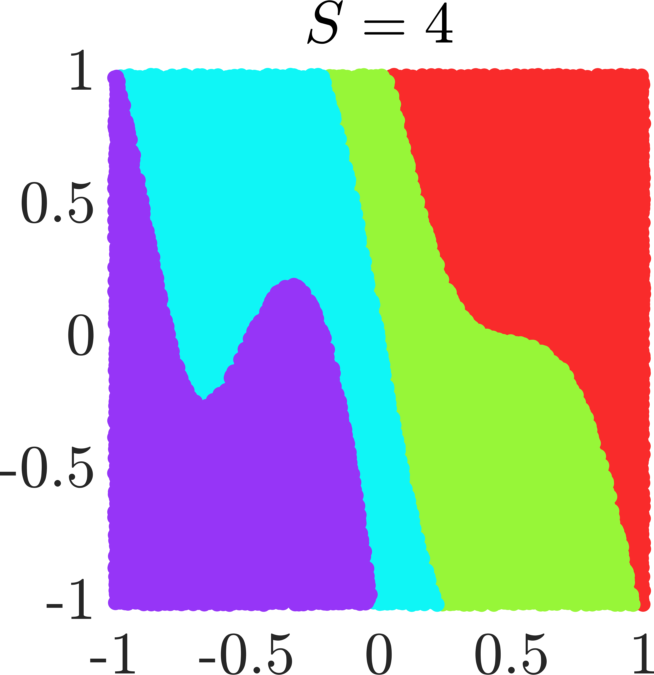}
\includegraphics{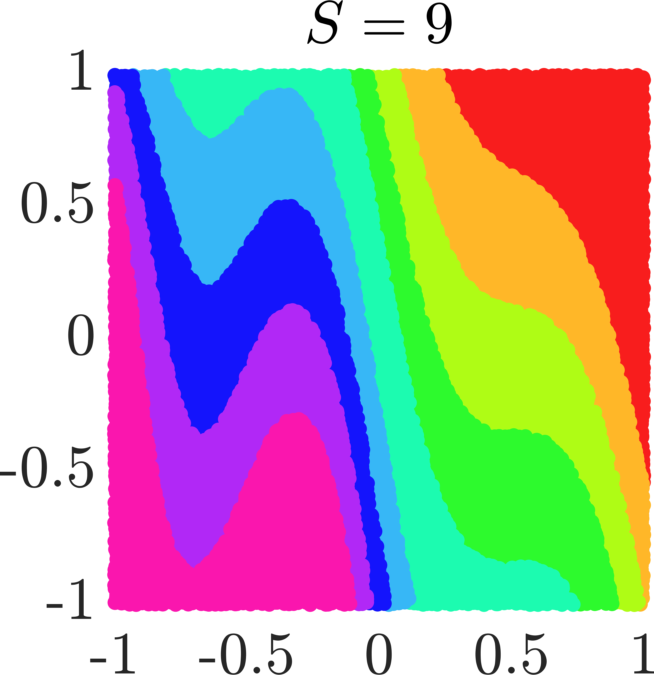}
\includegraphics{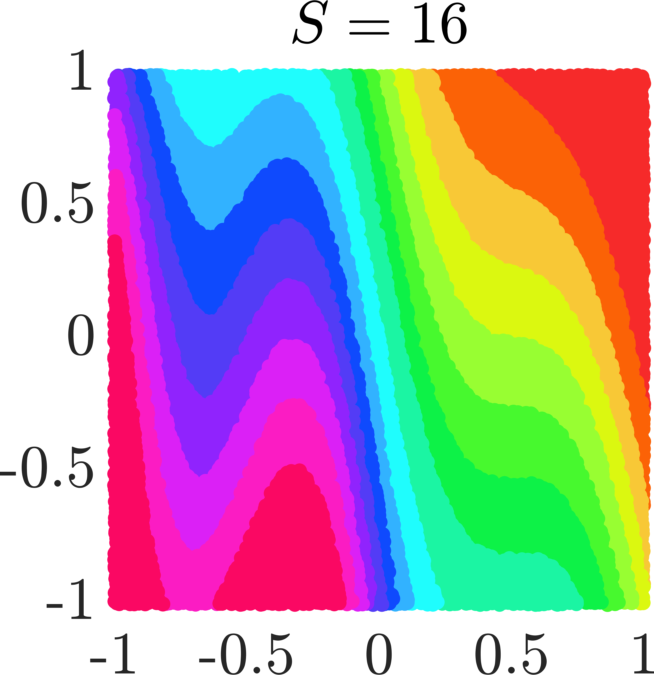}
\includegraphics{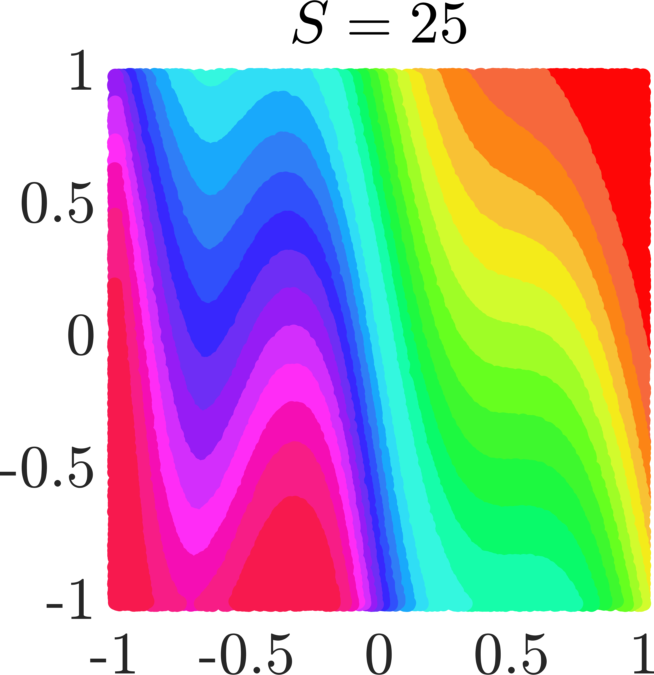}
\end{center}
\vspace{0.25cm}
\begin{center}
\includegraphics[scale=0.99]{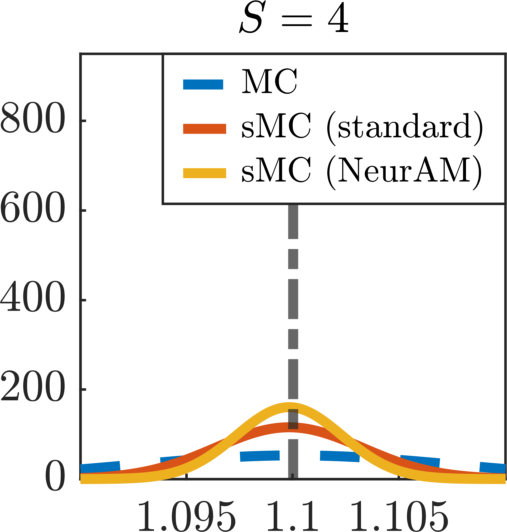}
\includegraphics[scale=0.99]{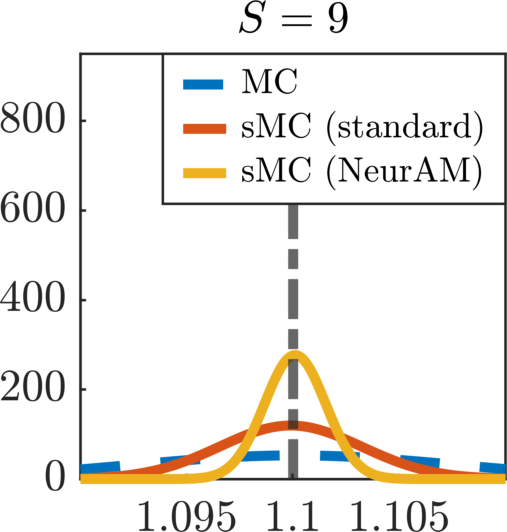}
\includegraphics[scale=0.99]{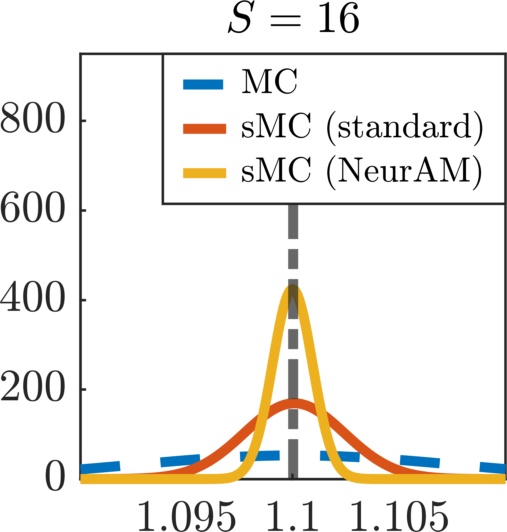}
\includegraphics[scale=0.99]{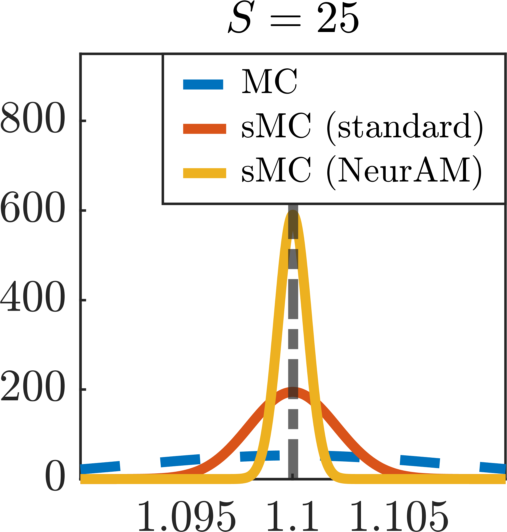}
\end{center}
\vspace{0.25cm}
\begin{center}
\begin{tabular}{c|cccc}
\toprule
sMC & $S = 4$ & $S = 9$ & $S = 16$ & $S = 25$ \\
\midrule
standard & 1.19e-5 (2.09e-01) & 1.11e-5 (1.95e-01) & 5.55e-6 (9.75e-02) & 4.19e-6 (7.36e-02) \\
NeurAM & 6.20e-6 (1.09e-01) & 2.08e-6 (3.66e-02) & 8.84e-7 (1.55e-02) & 4.56e-7 (8.01e-03) \\
\bottomrule
\end{tabular}
\end{center}
\vspace{0.25cm}
\begin{center}
\includegraphics{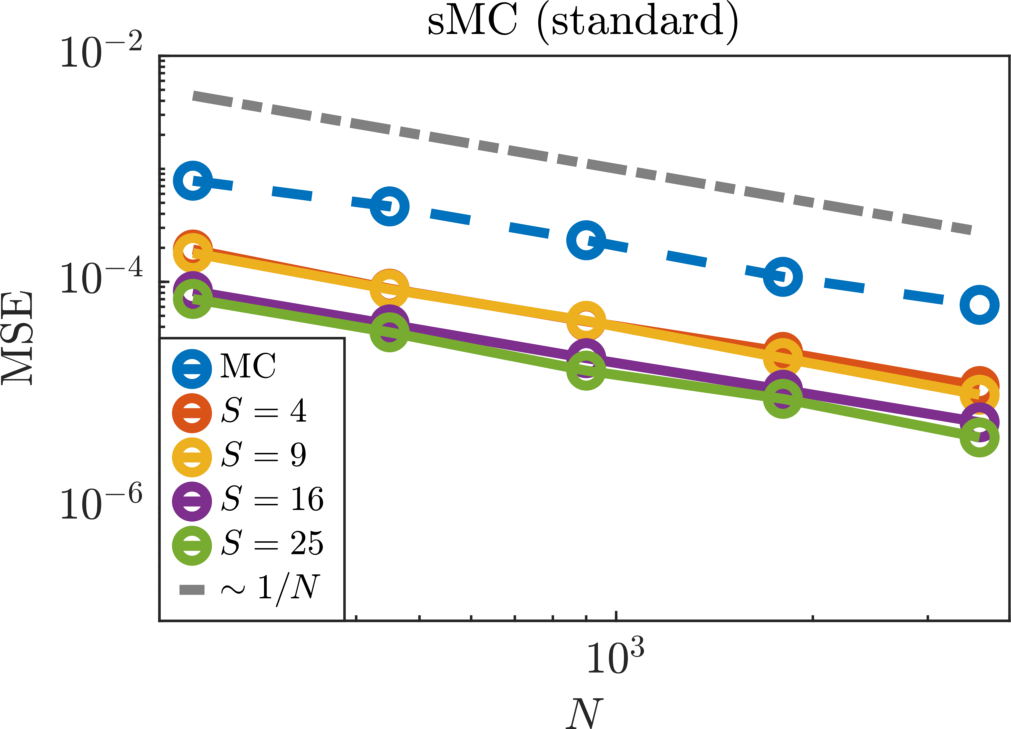} \hspace{0.1cm}
\includegraphics{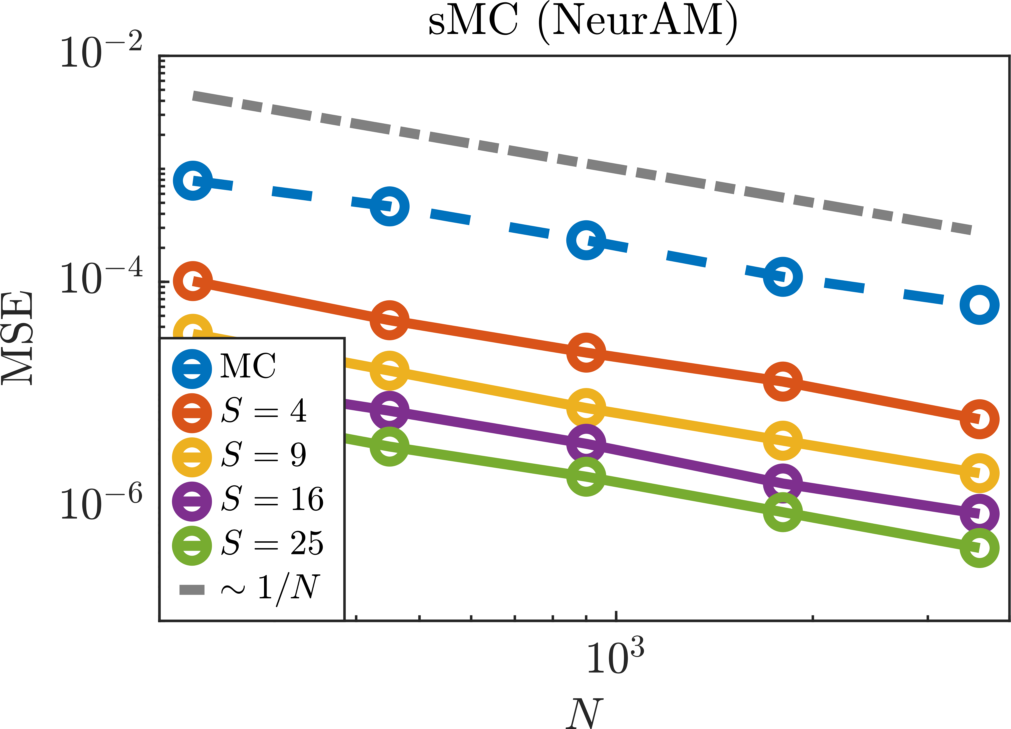}
\end{center}
\caption{\emph{TOP}: NeurAM-based stratification of the domain for the model $\Q_0$, varying the number of strata $S = 4, 9, 16, 25$. \emph{MIDDLE-TOP}: Comparison between standard Monte Carlo estimator $\widehat q_\MC$ (dashed line) and the stratified estimators $\widehat q_\sMC$ (solid line) with both standard and NeurAM-based stratification. The gray dash-dotted vertical line represents the exact value of the quantity of interest. \emph{MIDDLE-BOTTOM}: Mean squared error (MSE) of the estimators to be compared with the value 5.69e-5 for Monte Carlo obtained with $N=3600$. The numbers in parentheses denote the ratio between the MSE of each estimator and the Monte Carlo reference value. \emph{BOTTOM}: MSE as a function of the computational budget $N$.}
\label{fig:dependence_S}
\end{figure}

We consider the model $\Q_0$ in equation \eqref{eq:Q0} to study how the performance of the algorithm is affected by the choice of $M$, $K$ and $S$ under a uniform stratification, like in \cref{ex:uniform} in \cref{sec:NeurAM_stratification}. 
First, we fix the number of strata $S = 2$ and the computational budget $N = 1024$, and we vary the training dataset size $M = 5, 10, 50, 100$, and the number of latent samples $K = 10^3, 10^4, 10^5, 10^6$, respectively. 
Approximate distributions and mean squared errors (MSE) of various estimators are reported in \cref{fig:dependence_KM}, confirming an overall increase in performance with larger training datasets. 
In particular, the variance of the stratified estimator $\widehat q_\sMC$ decreases as $M$ increases, until it stabilizes. 
We remark that the number of samples is not a parameter that can be tuned, but it is usually constrained by the specific application, depending on the available model observations or computational budget. 

One could argue that, after NeurAM training, we could directly replace the model with its surrogate and compute as many evaluations as desired, since the surrogate model has a negligible computational cost. However, even if we can ideally get zero variance, this would lead to a bias in the estimation, as shown by the colored solid vertical lines in the plot. 
Moreover, this bias would increase by reducing the size $M$ of the training dataset.
Despite adopting a biased surrogate model, and even if an optimal variance reduction cannot be achieved due to a limited number of model evaluations, we can still use NeurAM to guide the stratification and achieve a reduction in variance. 
Specifically, this demonstrates that the surrogate model used in the NeurAM training does not need to be highly accurate to identify an effective nonlinear manifold for constructing a stratification that leads to variance reduction.

The hyperparameter $K$ is instead responsible for the bias in the estimation of the quantity of interest, and a smaller $K$ would generally result in a larger bias.
However, this is not a limitation of our approach as $K$ can always be chosen sufficiently large. 
In fact, sampling from the latent space is computationally cheap since, due to \cref{rem:strata_mu}, we can draw samples from $\mu_0$ and then apply the encoder $\E$ and the CDF $\F$, which have negligible computational costs. 
Nevertheless, from the table with the mean squared errors in \cref{fig:dependence_KM}, we notice that increasing $K$ is not always beneficial if $M$ is limited, and this is due to the fact that the one-dimensional manifold is not correctly identified. 
Yet, the improvement with respect to standard Monte Carlo is still evident, since its MSE is 1.93e-4.

We now consider a computational budget $N = 3600$, fix $M = 100$ and $K = 10^6$, and vary the number of strata $S = 4, 9, 16, 25$, while still adopting a uniform stratification.
In \cref{fig:dependence_S} we plot the stratified estimator for the different values of $S$, and compare Cartesian and NeurAM-based stratification.
The variance of the NeurAM-based estimator is significantly smaller for an increasing number of strata, yielding more precise estimates. This behavior can also be seen from the MSE as a function of the computational budget $N = 225, 450, 900, 1800, 3600$.
Moreover, we show the stratification produced by NeurAM for an increasing $S$, and observe how each stratum is bounded by the level sets of $\Q$ in \cref{fig:model}.
This property is crucial allowing NeurAM-based stratified estimators to scale to high-dimensional problems.

\subsection{Performance of the heuristic algorithm} \label{sec:num_heuristic}

\begin{figure}
\begin{center}
\includegraphics{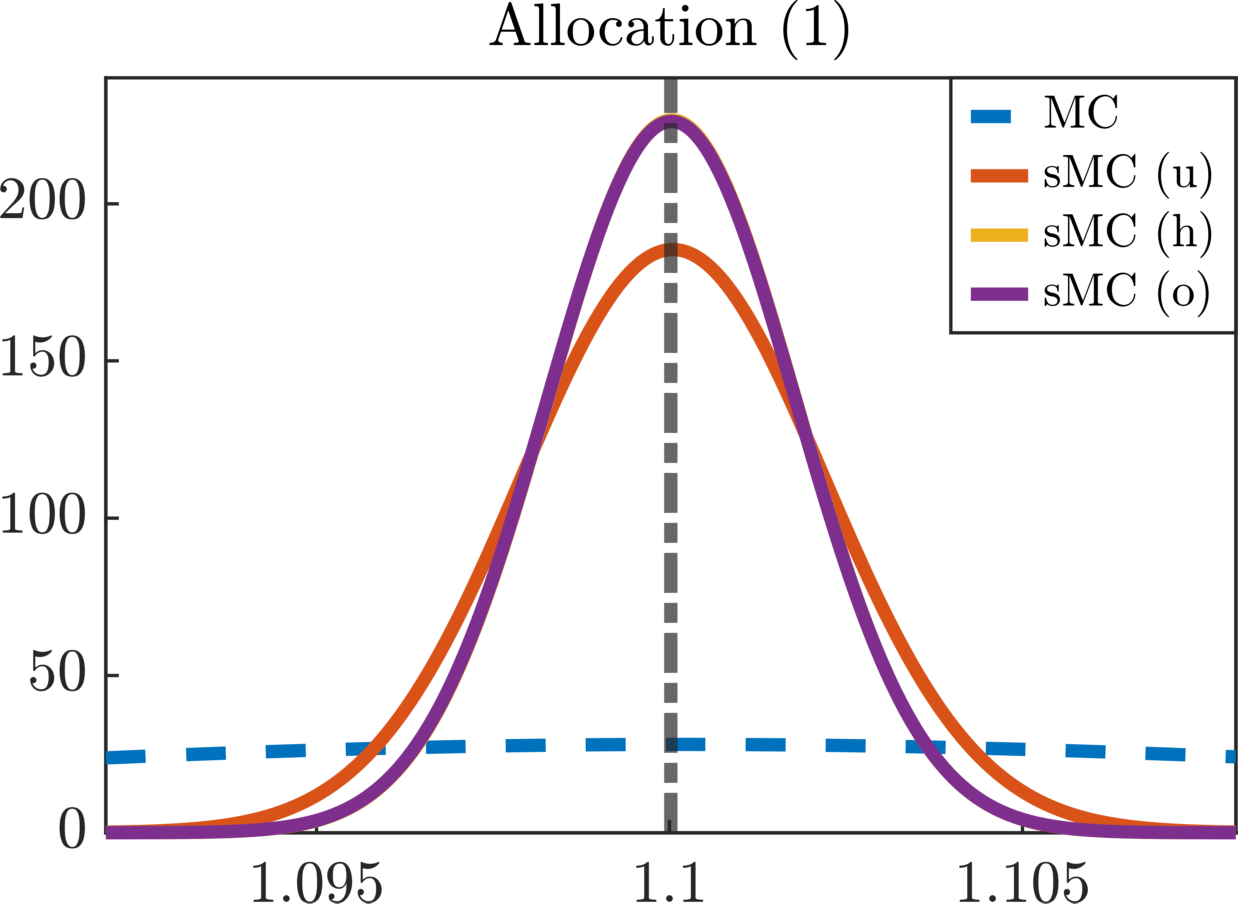} \hspace{1cm}
\includegraphics{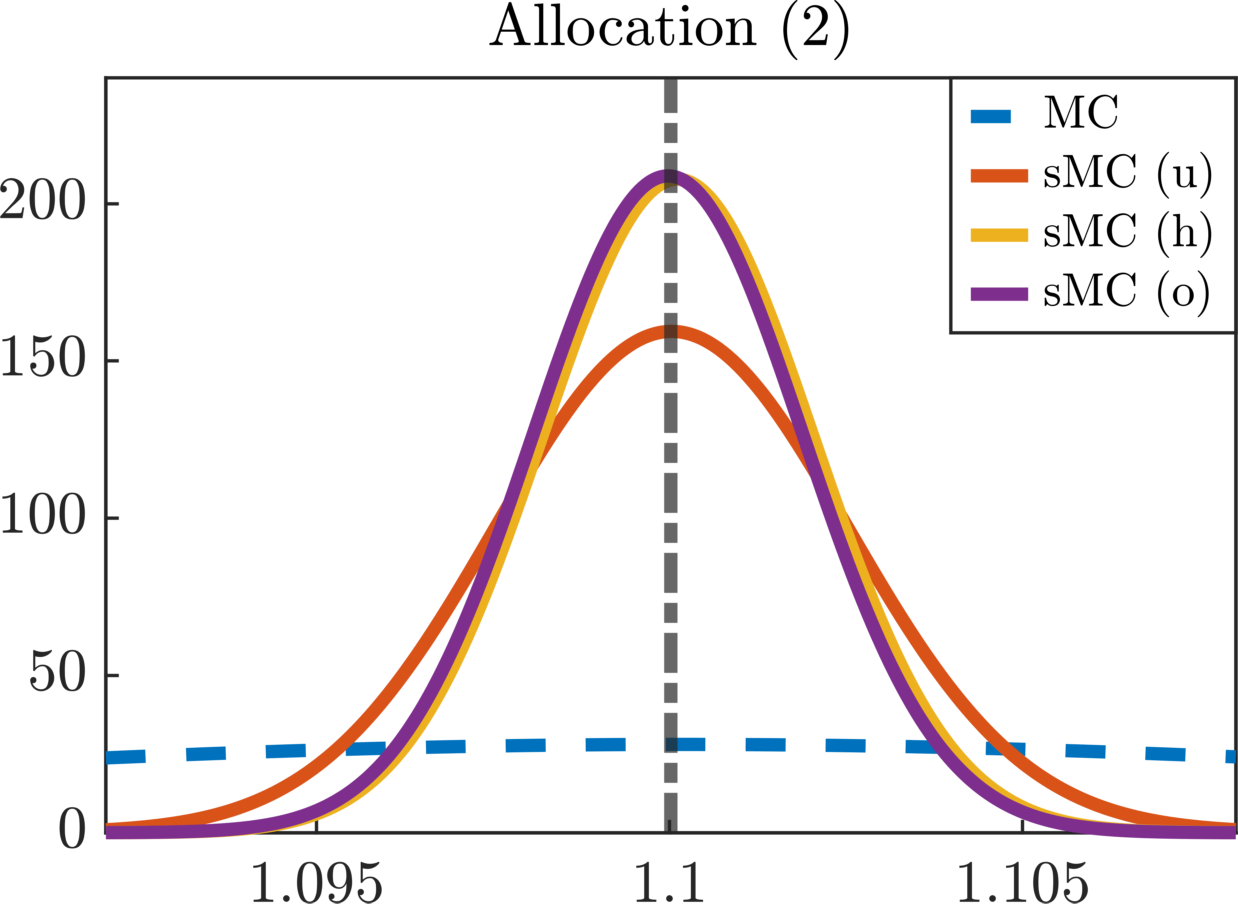}
\end{center}
\vspace{0.25cm}
\begin{center}
\begin{tabular}{c|ccc}
\toprule
& sMC (u) & sMC (h) & sMC(o) \\
\midrule
Allocation (1) & 4.62e-6 (2.29e-02) & 3.10e-6 (1.53e-02) & 3.11e-6 (1.54e-02) \\
Allocation (2) & 6.26e-6 (3.10e-02) & 3.69e-6 (1.83e-02) & 3.65e-6 (1.81e-02) \\
\bottomrule
\end{tabular}
\end{center}
\vspace{0.25cm}
\begin{center}
\includegraphics[scale=0.99]{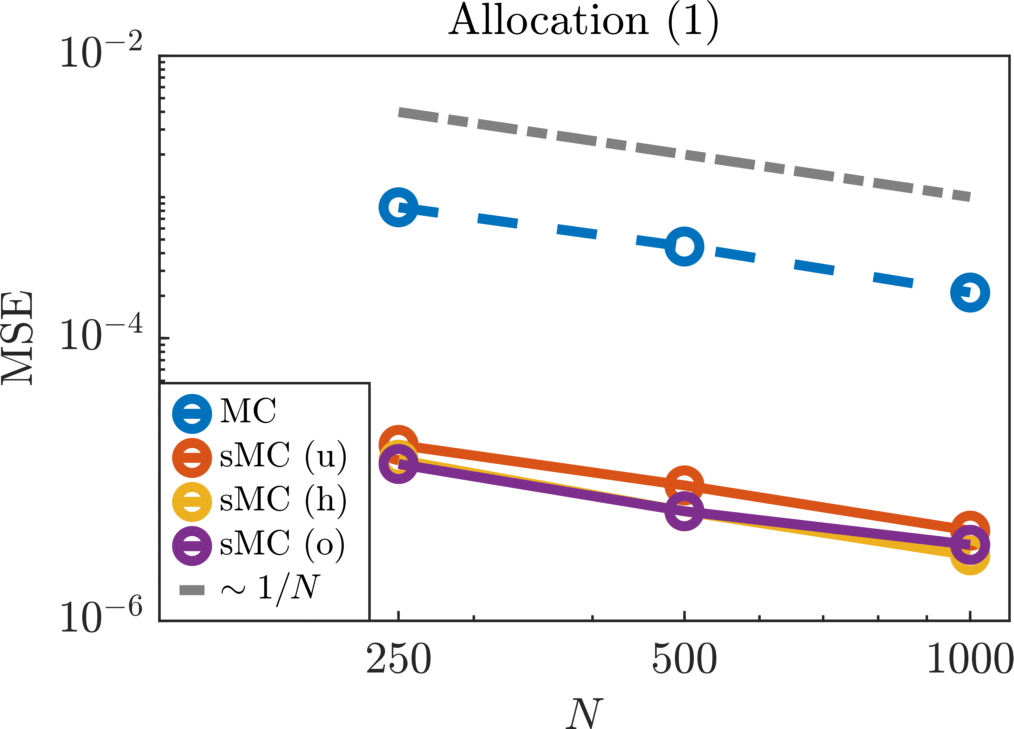} \hspace{0.1cm}
\includegraphics[scale=0.99]{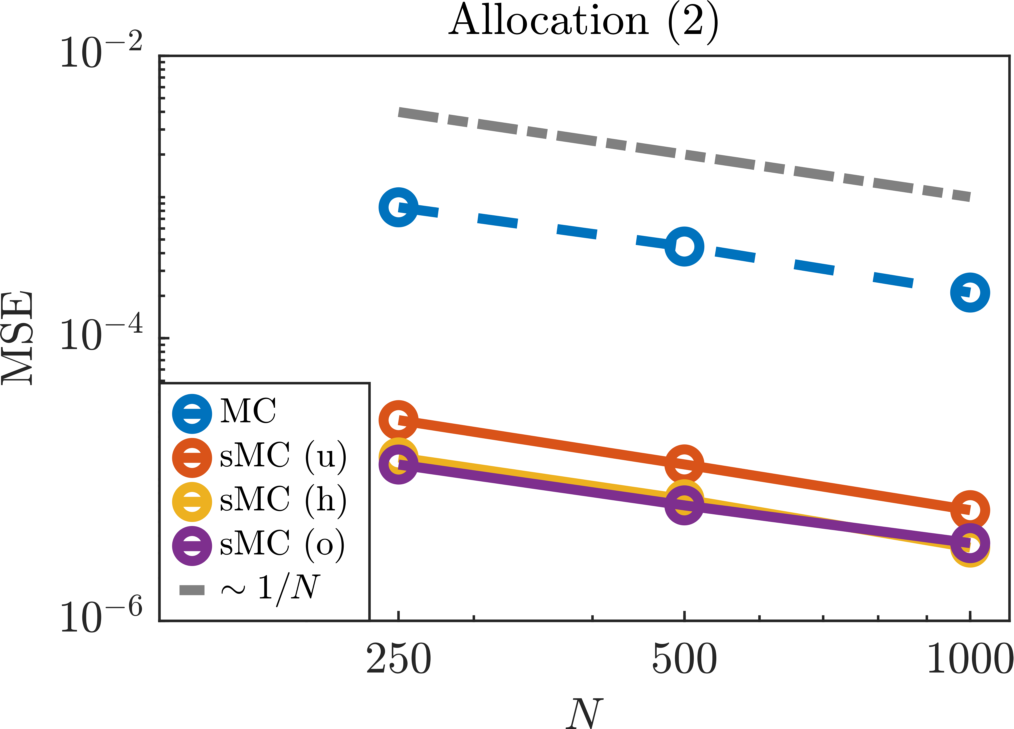}
\end{center}
\caption{\emph{TOP}: Comparison between standard Monte Carlo estimator $\widehat q_\MC$ (dashed line) and the proposed estimator $\widehat q_\sMC$ (solid line), for different allocation strategies and stratification approaches: uniform (u), halved (h), optimal (o). The gray dash-dotted vertical line represents the exact value of the quantity of interest. \emph{MIDDLE}: Mean squared error (MSE) of the estimators to be compared with the value 2.02e-4 for Monte Carlo obtained with $N=1000$. The numbers in parentheses denote the ratio between the MSE of each estimator and the Monte Carlo reference value. \emph{BOTTOM}: MSE as a function of the computational budget $N$.}
\label{fig:heuristic}
\end{figure}

We now investigate possible improvements achieved by the heuristic algorithm introduced in \cref{sec:heuristic}. 
We fix $M = 100$, $K = 10^6$, $S = 10$, assume a computational budget $N = 1000$, and test two allocation strategies, as well as two alternatives to iteratively refine the strata.
The comparison with standard Monte Carlo and Monte Carlo with Cartesian stratification is shown in \cref{fig:heuristic}.
We observe that the heuristic algorithm is always able to outperform uniform stratification, but we do not notice a significant difference between following a bisection approach, i.e., dividing in two equal parts the interval that contributes the most to the final variance, and selecting the optimal new stratum at each step. This behavior can also be seen from the MSE as a function of the computational budget $N = 250, 500, 1000$.
Therefore, in order to limit the additional computational overhead, we suggest to use the former approach. 
We finally remark that, as predicted by theory (see equation \eqref{eq:bound_variances}), the optimal allocation (1) gives a variance which is always smaller than the one produced by a proportional allocation (2).

\subsection{Comparison with active subspaces} \label{sec:num_AS}

\begin{figure}
\begin{center}
\setlength{\tabcolsep}{0.1cm}
\begin{tabular}{cc}
\begin{tabular}{cc}
\includegraphics{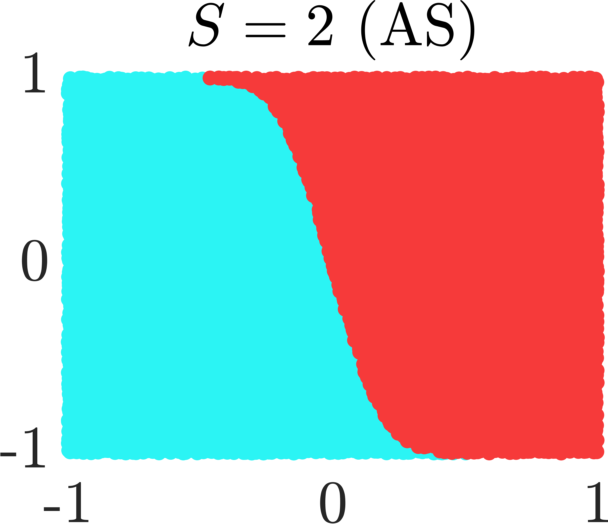} &
\includegraphics{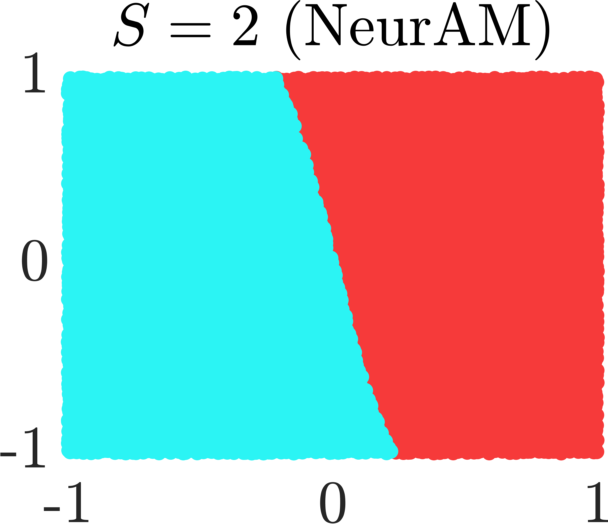}
\end{tabular} &
\begin{tabular}{cc}
\includegraphics{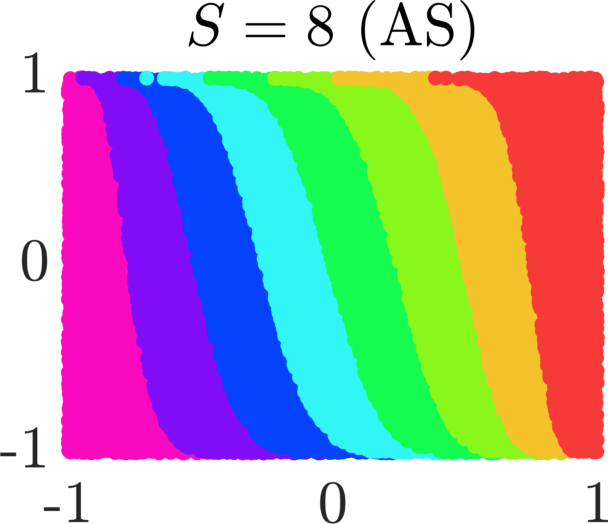} &
\includegraphics{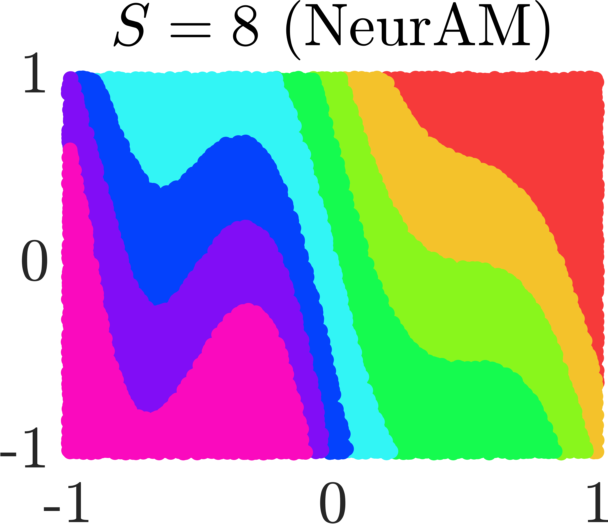}
\end{tabular} \\
\includegraphics{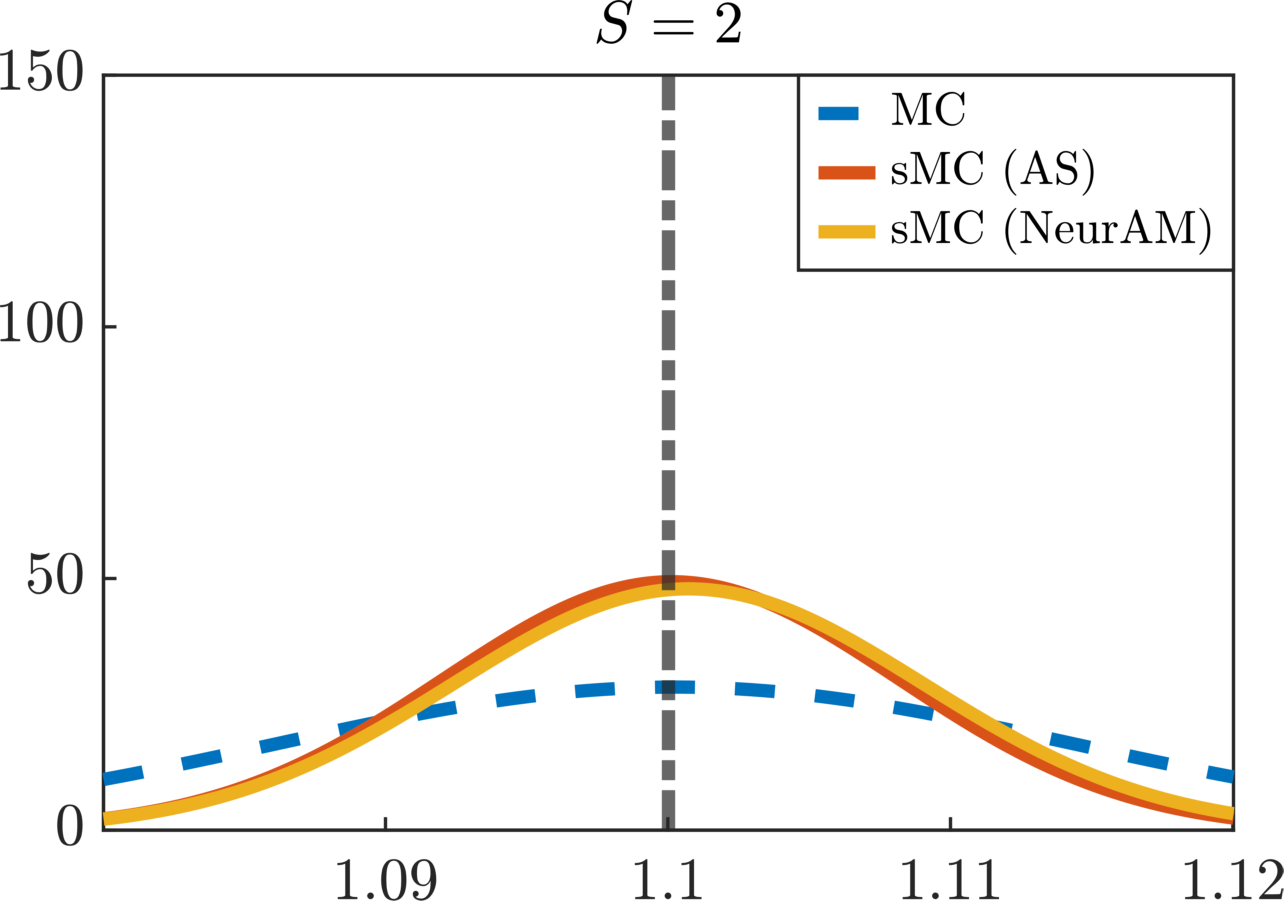} &
\includegraphics{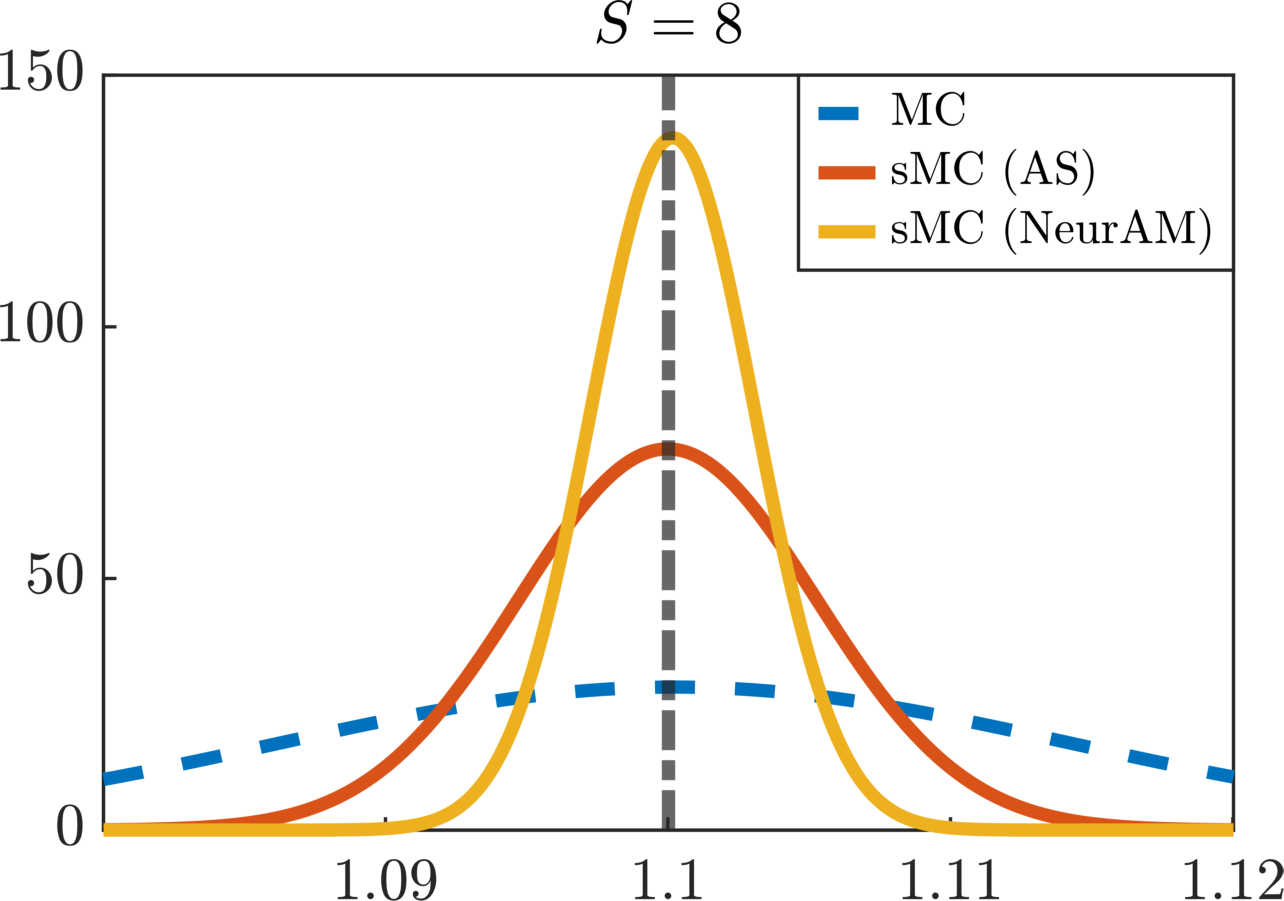}
\end{tabular}
\end{center}
\vspace{0.25cm}
\begin{center}
\begin{tabular}{c|cc}
\toprule
& $S = 2$ & $S = 8$ \\
\midrule
sMC (AS) & 6.53e-5 (3.31e-01) & 2.77e-5 (1.41e-01) \\
sMC (NeurAM) & 6.97e-5 (3.54e-01) & 8.42e-6 (4.27e-02) \\
\bottomrule
\end{tabular}
\end{center}
\vspace{0.25cm}
\begin{center}
\includegraphics{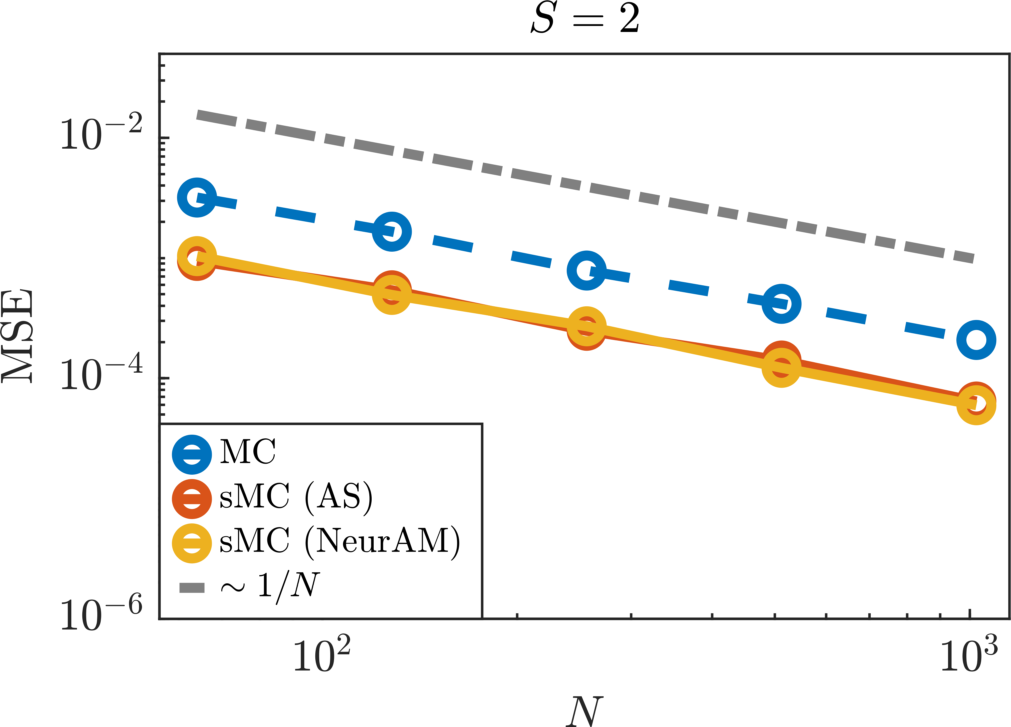} \hspace{0.1cm}
\includegraphics{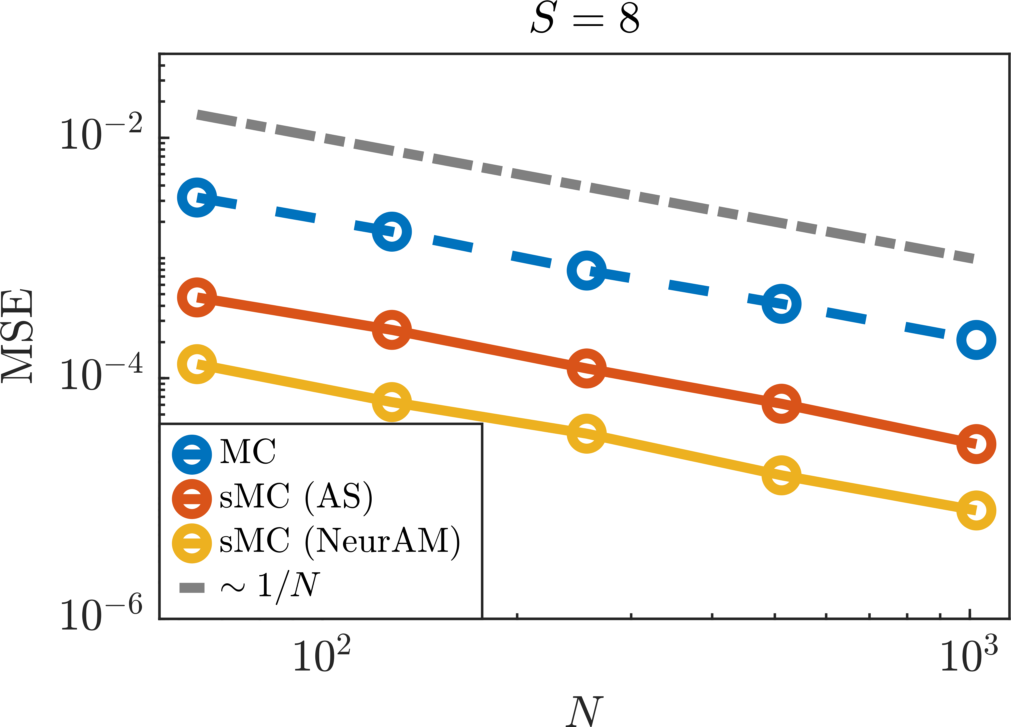}
\end{center}
\caption{\emph{TOP}: AS and NeurAM-based stratifications of the domain for the model $\Q_0$, varying the number of strata $S = 2, 8$. \emph{MIDDLE-TOP}: Comparison between standard Monte Carlo estimator $\widehat q_\MC$ (dashed line) and stratified Monte Carlo estimator $\widehat q_\sMC$ (solid line) with both AS and NeurAM-based stratification, varying the number of strata $S = 2$ (left) and $S = 8$ (right). The gray dash-dotted vertical line represents the exact value of the quantity of interest. \emph{MIDDLE-BOTTOM}: Mean squared error (MSE) of the estimators to be compared with the value 1.97e-4 for Monte Carlo obtained with $N=1024$. The numbers in parentheses denote the ratio between the MSE of each estimator and the Monte Carlo reference value. \emph{BOTTOM}: MSE as a function of the computational budget $N$.}
\label{fig:AS}
\end{figure}

NeurAM has the advantages of being nonlinear and entirely data-driven, requiring no knowledge of the model gradient, but it is not the only available option to improve stratified sampling. In this section, we demonstrate that stratified sampling can be combined with other dimensionality reduction techniques. However, we also show that NeurAM outperforms linear state-of-the-art approaches such as the active subspace method \cite{CDW14}. Inspired by \cite{ZGS24b}, we build a map from the original domain to the unit interval as follows. Let $\mathcal G \colon \R^d \to \R^d$ be a transformation that maps a standard Gaussian $\mathcal N(0,I)$ into the input distribution $\mu$, i.e. $\mathcal G_\# \mathcal N(0,I) = \mu$, and define the reparameterized model
\begin{equation}
\widetilde\Q(z) = \Q(\mathcal G(z)).
\end{equation}
For this test case, since $\mu_0 = \mathcal U([-1,1]^2)$, we have
\begin{equation}
\mathcal G_0(z) = \erf \left( \frac{z}{\sqrt 2} \right),
\end{equation}
where the function $\erf$ is evaluated componentwise, otherwise the map can be computed using, e.g., normalizing flows \cite{KPB21}. Then, let $B$ be the matrix
\begin{equation}
B = \Ex^{\mathcal N(0,I)} \left[ \nabla \widetilde\Q(Z) \nabla \widetilde\Q(Z)^\top \right],
\end{equation}
which is symmetric positive semidefinite and therefore has positive eigenvalues. We denote by $v$ the eigenvector corresponding to the largest eigenvalue, representing the one-dimensional active subspace, and normalize it to have unit norm. We note that the map $\mathcal G$ is important because Gaussian distributions are preserved under linear transformations and, in particular, if $z \sim \mathcal N(0,I)$ then $v^\top z \sim \mathcal N(0,1)$. Therefore, a point $x \in \mathbb D$ can be mapped into the unit interval $[0,1]$ through the one-dimensional AS, by applying the inverse CDF of the standard Gaussian distribution to $v^\top \mathcal G^{-1}(x)$. Specifically, using the notation of \cref{sec:NeurAM_stratification}, equation \eqref{eq:Ds_def} now reads
\begin{equation} \label{eq:Ds_def_AS}
D_s^{\mathrm{AS}} = \left\{ x \in \mathbb D \colon \frac12 \left( \erf \left( \frac{v^\top \mathcal G^{-1}(x)}{\sqrt 2} \right) + 1 \right) \in A_s \right\}.
\end{equation}
In \cref{fig:AS} we compare the stratification provided by AS with the NeurAM-based stratification for a uniform division of the unit interval into $S = 2$ and $S = 8$ strata, and a computational budget of $N = 1024$. The distribution of the corresponding stratified Monte Carlo estimators with proportional allocation is also shown and compared to that of standard Monte Carlo, which is consistently outperformed. This behavior can also be seen from the MSE as a function of the computational budget $N = 64, 128, 256, 512, 1024$. While the results are comparable for $S = 2$, we observe that the estimator based on nonlinear dimensionality reduction achieves significantly greater variance reduction than linear techniques when the number of strata increases, e.g., $S = 8$. This improvement is primarily due to the greater expressiveness of nonlinear methods, which are better able to follow the contour lines of nonlinear models.

\subsection{Stratified multifidelity estimators} \label{sec:num_multifidelity}

\begin{figure}
\begin{center}
\includegraphics{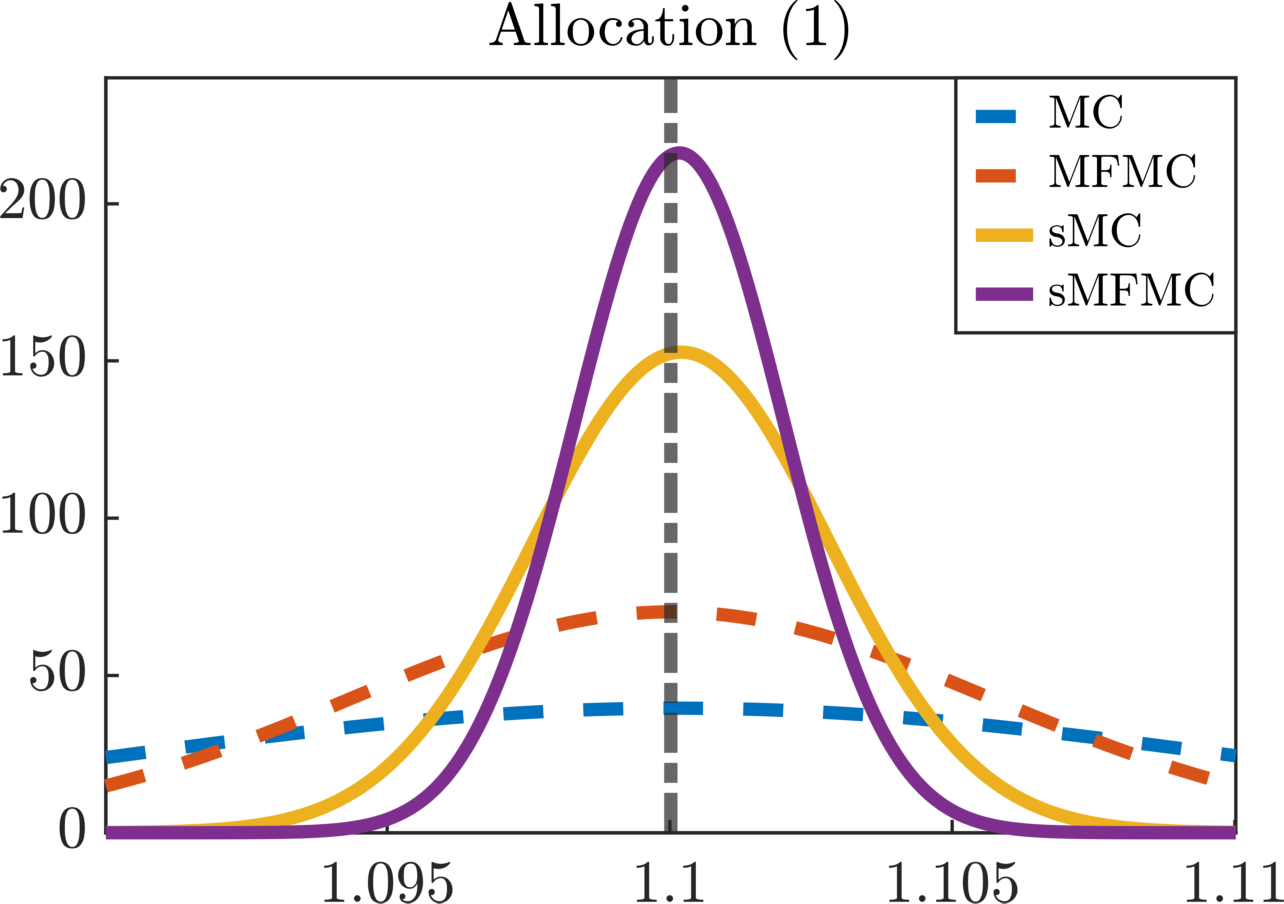} \hspace{0.5cm}
\includegraphics{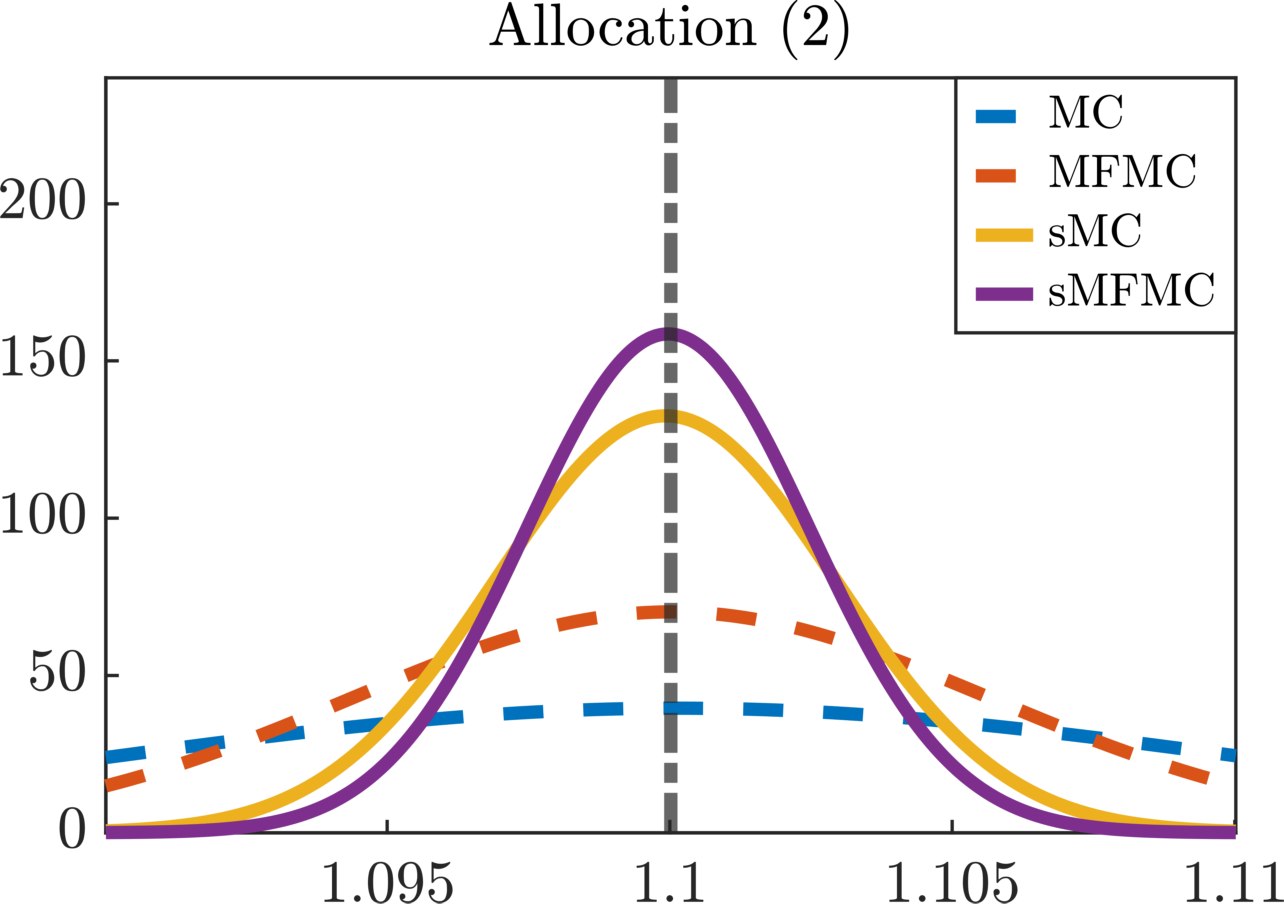}
\end{center}
\vspace{0.25cm}
\vspace{0.25cm}
\begin{center}
\begin{tabular}{c|cc}
\toprule
& sMC & sMFMC \\
\midrule
Allocation (1) & 6.84e-6 (6.77e-02) & 3.43e-6 (3.40e-02) \\
Allocation (2) & 9.06e-6 (8.97e-02) & 6.32e-6 (6.26e-02) \\
\bottomrule
\end{tabular}
\end{center}
\vspace{0.25cm}
\begin{center}
\includegraphics{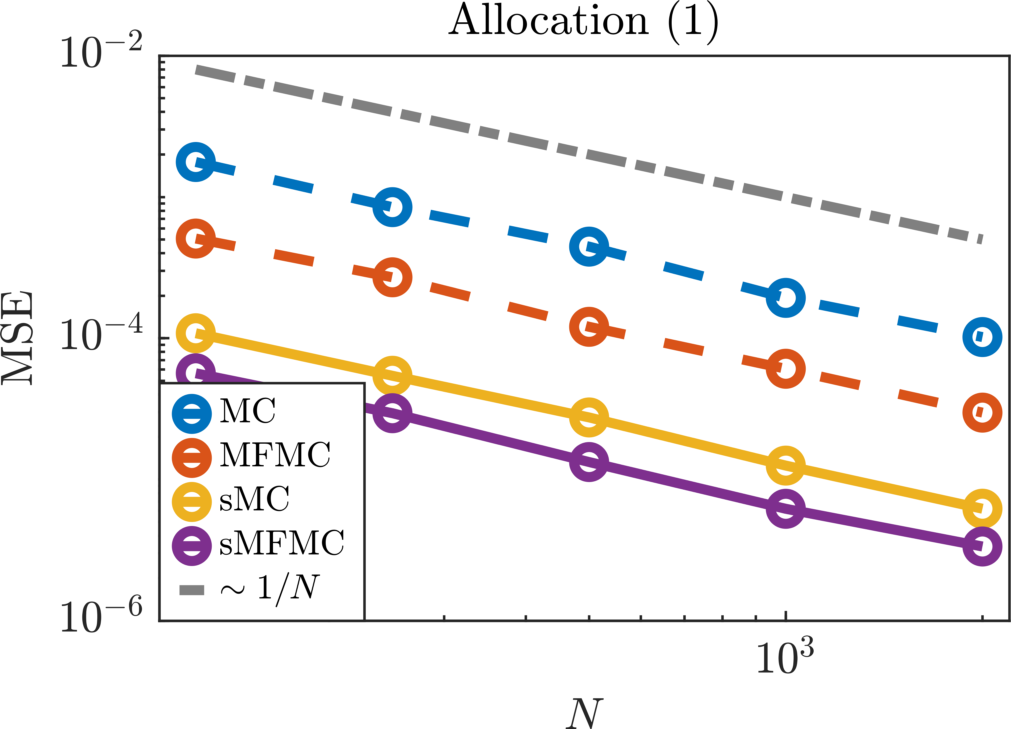} \hspace{0.1cm}
\includegraphics{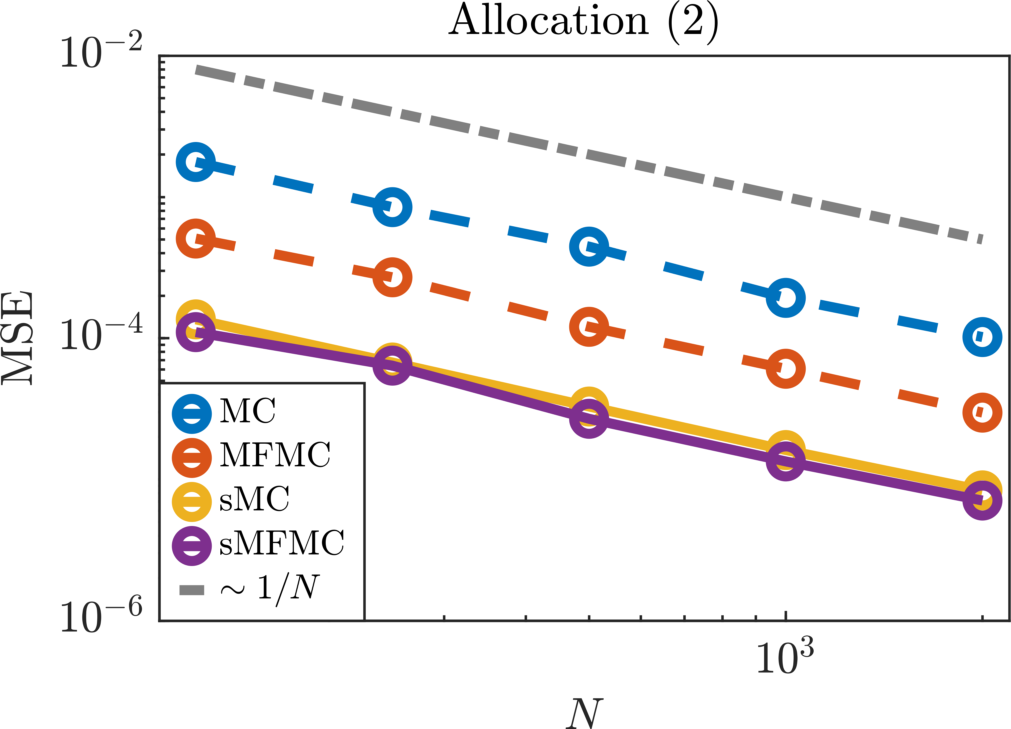}
\end{center}
\caption{\emph{TOP}: Comparison between standard (multifidelity) Monte Carlo estimators $\widehat q_\MC, \widehat q_\MFMC$ (dashed line) and the proposed estimators $\widehat q_\sMC, \widehat q_\sMFMC$ (solid line), for different allocation strategies. The gray dash-dotted vertical line represents the exact value of the quantity of interest. \emph{MIDDLE}: Mean squared error (MSE) of the estimators to be compared with the value 1.01e-4 for Monte Carlo. The numbers in parentheses denote the ratio between the MSE of each estimator and the Monte Carlo reference value. Standard multifidelity Monte Carlo achieves an MSE of 3.22e-5 (3.19e-01). \emph{BOTTOM}: MSE as a function of the computational budget $N$.}
\label{fig:multifidelity}
\end{figure}

In this section, we test the performance of estimators combining NeurAM stratification and multifidelity variance reduction, as discussed in \cref{sec:multifidelity}.
Let $\Q^\HF = \Q_0$, and consider a low-fidelity model of the form
\begin{equation}
Q^\LF(x) = e^{0.01x_1 + 0.99x_2} + 0.15\sin(3\pi x_2),
\end{equation}
for which we assume a cost ratio $w = 0.01$, to reflect cost differences in realistic applications.
This poorly correlated low-fidelity model has already been used in \cite{GEG18,ZGS24a,ZGS24b,ZGS25}. 
Following \cite[Section 3.2]{ZGS25}, we then train NeurAM for both models using $M = 100$, and re-parameterize the low-fidelity model using
\begin{equation}
\Q^\LF(x) = Q^\LF(\D^\LF((\F^\LF)^{-1}(\F^\HF(\E^\HF(x))))),
\end{equation} 
resulting in higher correlations with respect to the high-fidelity model. 
In \cref{fig:multifidelity} we compare (stratified) Monte Carlo estimators with (stratified) multifidelity Monte Carlo estimators for both the allocation strategies, assuming a computational budget $N = 2,000$ and using a uniform stratification with $S = 5$ strata. 
First, we observe that stratification is always beneficial, and stratified estimators outperform multifidelity estimators. This behavior can also be seen from the MSE as a function of the computational budget $N = 125, 250, 500, 1000, 2000$.
Then, we notice that leveraging stratification for multifidelity estimators allows us to obtain a significant variance reduction. 
Finally, we remark that, when using the optimal allocation strategy (1), the improvement with respect to single-fidelity estimators is greater. 
This is in agreement with \cref{rem:variance_MF,pro:variance_MF}, where a weaker sufficient condition is provided for allocation (1) to achieve variance reduction.

\subsection{Higher-dimensional models} \label{sec:num_highDimensional}

\begin{figure}
\begin{center}
\includegraphics{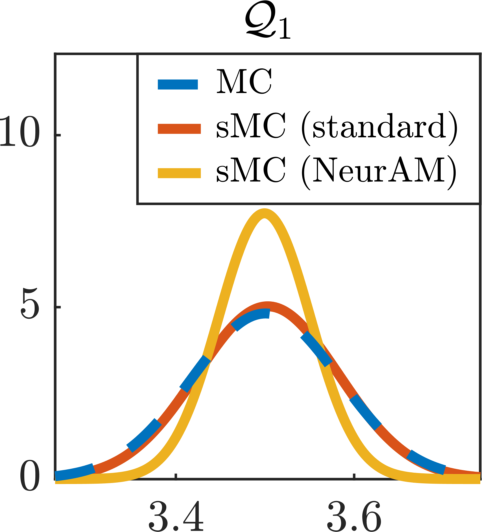}
\includegraphics{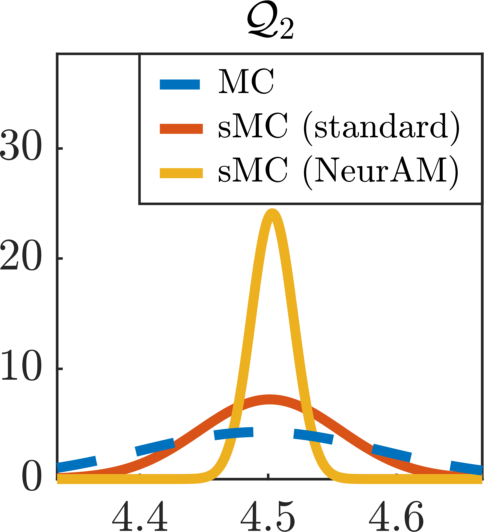}
\includegraphics{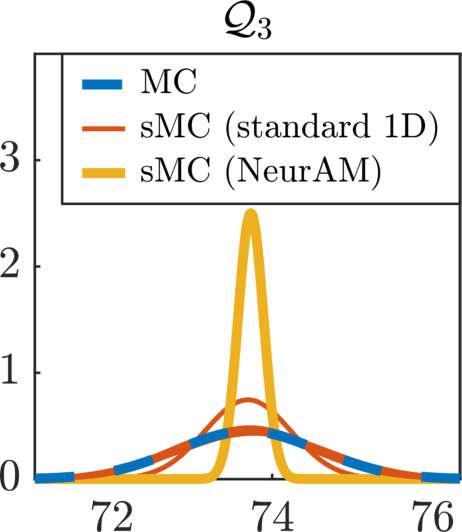}
\includegraphics{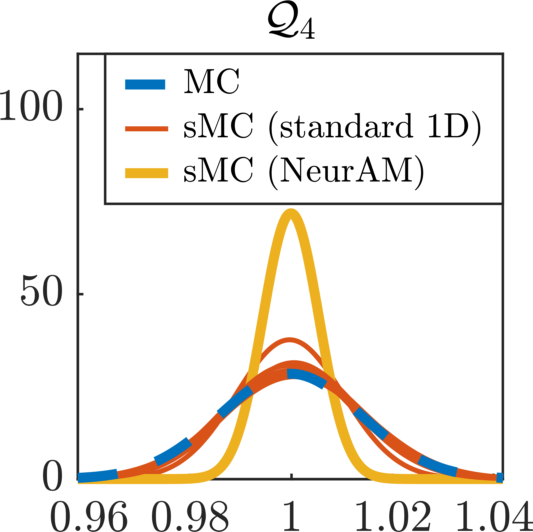}
\end{center}
\vspace{0.25cm}
\begin{center}
\begin{tabular}{c|ccc}
\toprule
& MC & sMC (standard -- best) & sMC (NeurAM)  \\
\midrule
$\Q_1$ & 6.48e-03 (1) & 6.31e-03 (9.74e-01) & 2.36e-03 (3.65e-01) \\
$\Q_2$ & 8.22e-03 (1) & 2.91e-03 (3.54e-01) & 8.96e-05 (1.09e-02) \\
$\Q_3$ & 7.81e-01 (1) & 2.87e-01 (3.67e-01) & 2.55e-02 (3.26e-02) \\
$\Q_4$ & 1.96e-04 (1) & 1.12e-04 (5.73e-01) & 3.08e-05 (1.57e-01) \\
\bottomrule
\end{tabular}
\end{center}
\vspace{0.25cm}
\begin{center}
\includegraphics[scale=0.96]{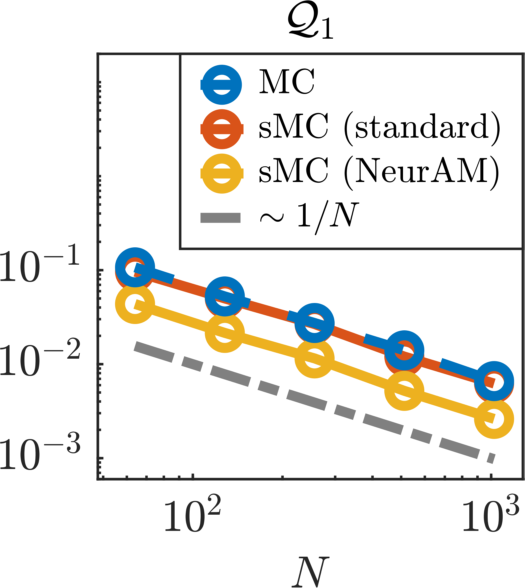}
\includegraphics[scale=0.96]{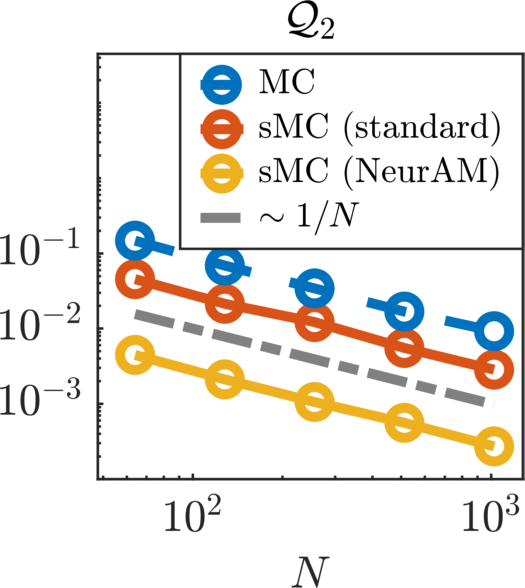}
\includegraphics[scale=0.96]{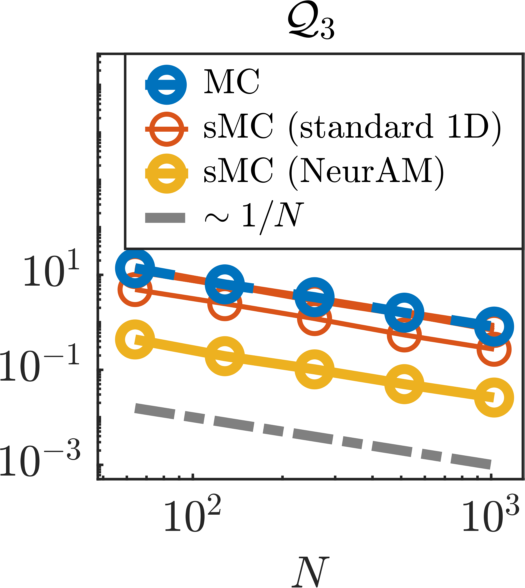}
\includegraphics[scale=0.96]{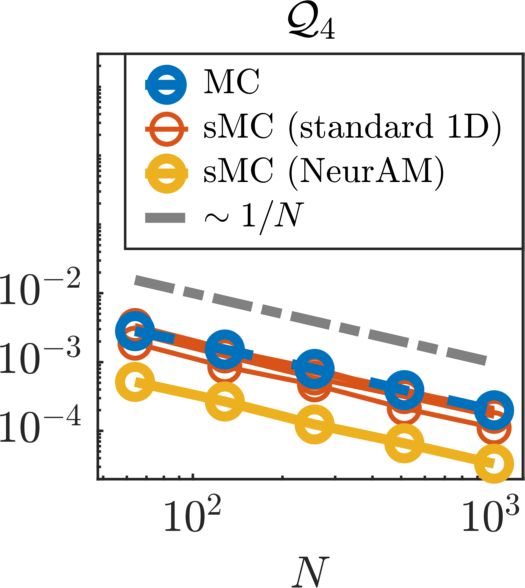}
\end{center}
\caption{\emph{TOP:} Comparison between standard Monte Carlo estimator $\widehat q_\MC$ (dashed line) and the stratified estimators $\widehat q_\sMC$ (solid line) with both standard -- either a regular Cartesian grid ($\Q_1$ and $\Q_2$) or uniform stratification along one dimension ($\Q_3$ and $\Q_4$) -- and NeurAM-based stratification, for the four models in \cref{sec:num_highDimensional}. \emph{MIDDLE}: Variance of the estimators obtained with $N=2000$. For the uniform stratification along one dimension of the models $\Q_3$ and $\Q_4$, the best value is considered. The numbers in parentheses denote the ratio between the variance of each estimator and the Monte Carlo reference value. \emph{BOTTOM}: Variance as a function of the computational budget $N$.}
\label{fig:highDimensional}
\end{figure}

After investigating the convergence of the proposed approach and its sensitivity to hyperparameters, we now focus on how these properties scale to higher dimensions.
Consider the following models:
\begin{itemize}[leftmargin=*] 
\item $\Q_1 \colon \R^3 \to \R$ proposed in \cite{IsH90} which exhibits strong nonlinearity and nonmonotonicity
\begin{equation}
\Q_1(x) = \sin(\pi x_1) + 7\sin(\pi x_2)^2 + 0.1 \pi x_3^4 \sin(\pi x_1);
\end{equation}
\item $\Q_2 \colon \R^4 \to \R$ employed in \cite{GCS17} which models the average velocity of a steady, incompressible, and laminar flow of an electrically conducting fluid between two infinite parallel plates in the presence of a uniform magnetic field (see Hartmann problem, e.g. \cite{SPC16})
\begin{equation}
\Q_2(x) = - \frac{x_2 x_3}{x_4^2} \left( 1 - \frac{x_4}{\sqrt{x_3 x_1}} \coth \left( \frac{x_4}{\sqrt{x_3 x_1}} \right) \right);
\end{equation}
\item $\Q_3 \colon \R^8 \to \R$ proposed in \cite{HaG83} as a model for the flow of water through a borehole
\begin{equation}
\Q_3(x) = \frac{2\pi x_3 (x_4 - x_6)}{\log \left( \frac{x_2}{x_1} \right) \left(1 + \frac{x_3}{x_5} + \frac{2 x_7 x_3}{\log \left( \frac{x_2}{x_1} \right) x_1^2 x_8} \right)};
\end{equation}
\item $\Q_4 \colon \R^{10} \to \R$ which is a modification of the so-called $g$-function
\begin{equation}
\Q_4(x) = \prod_{i=1}^{10} \frac{2\abs{x_i} + i}{1 + i}.
\end{equation}
\end{itemize}
Moreover, let the corresponding input probability distributions be:
\begin{itemize}[leftmargin=*]
\item $\mu_1 = \mathcal U([-1,1]^3)$;
\item $\mu_2 = \log \mathcal U([0.05, 0.2]) \otimes \log \mathcal U([0.5, 3]) \otimes \log \mathcal U([0.5, 3]) \otimes \log \mathcal U([0.1, 1])$;
\item $\mu_3$ given by the following product measure
\begin{equation}
\begin{aligned}
\mu_3 &= \mathcal N(0.10, 0.0161812^2) \otimes \log \mathcal N(7.71, 1.0056^2) \otimes \mathcal U([63070, 115600]) \otimes \mathcal U([990, 1110]) \\
&\quad \otimes \mathcal U([63.1, 116]) \otimes \mathcal U([700, 820]) \otimes \mathcal U([1120, 1680]) \otimes \mathcal U([9855, 12045]);
\end{aligned}
\end{equation}
\item $\mu_4 = \mathcal U([-1,1]^{10})$,
\end{itemize}
where $\log \mathcal U$ and $\log \mathcal N$ denote the log-uniform and log-normal distributions, respectively. 
In \cref{fig:highDimensional} we compare standard and stratified Monte Carlo estimators for all the models above, assuming a computational budget $N = 1024$ and training NeurAM using a dataset of size $M = 1024$. 
The uniform NeurAM-based stratification is computed using $S = 16$ strata, except for $\Q_1$ for which we use $S = 8$. 
In the first two test cases, characterized by a relatively low input dimensionality, we also plot the estimator based on a Cartesian stratification on a regular multidimensional grid. 
Note that a Cartesian grid is not sustainable in high dimensions. In fact, even choosing two strata per dimension, in, e.g., 10 dimensions, this would lead to $2^{10} = 1024$ strata in total, saturating the computational budget.
Therefore, for the last two test cases, in Figure~\ref{fig:highDimensional} we also plot stratified Monte Carlo estimators obtained by partitioning each dimension uniformly, one at a time. We note that the performance in these cases is poor, with results being roughly equivalent across all dimensions.
We observe that our methodology outperforms standard stratification and provides effective variance reduction even for high dimensional problems. This behavior can also be seen from the plot of the variance as a function of the computational budget $N = 64, 128, 256, 512, 1024$.
Regarding the model $\Q_4$, we finally remark that the NeurAM surrogate is inaccurate, giving an approximation error of $\sim 23 \%$. Nevertheless, the variance reduction is significant, confirming that a highly accurate surrogate model is not essential to improve the estimation even in high dimensions.

\subsection{Comparison with other variance reduction techniques} \label{sec:num_other}

\begin{figure}
\begin{center}
\begin{tabular}{ccc}
\includegraphics{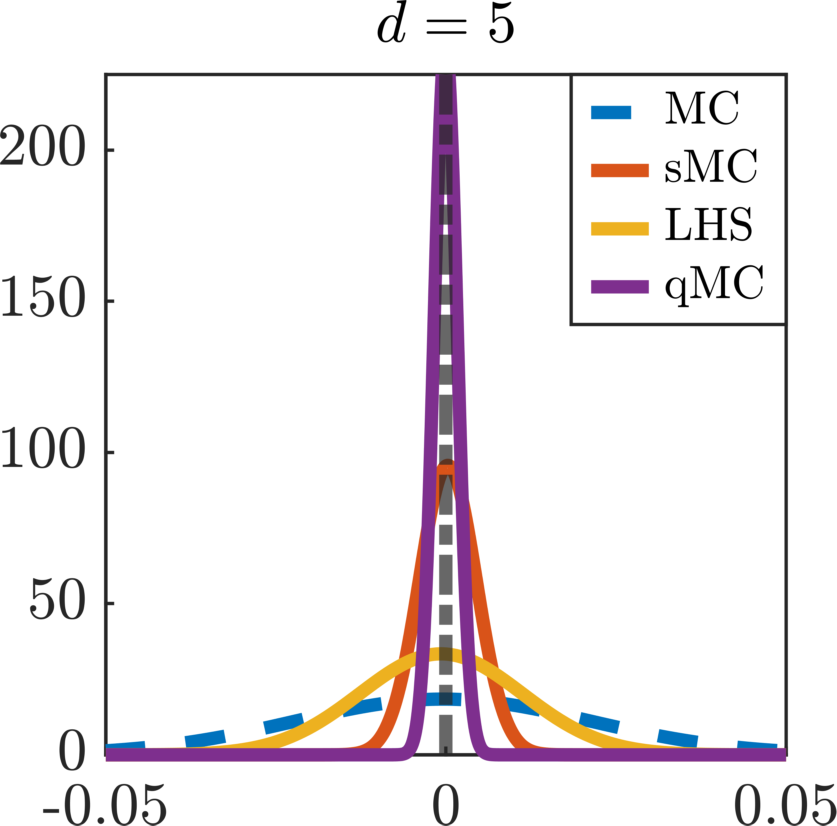} &
\includegraphics{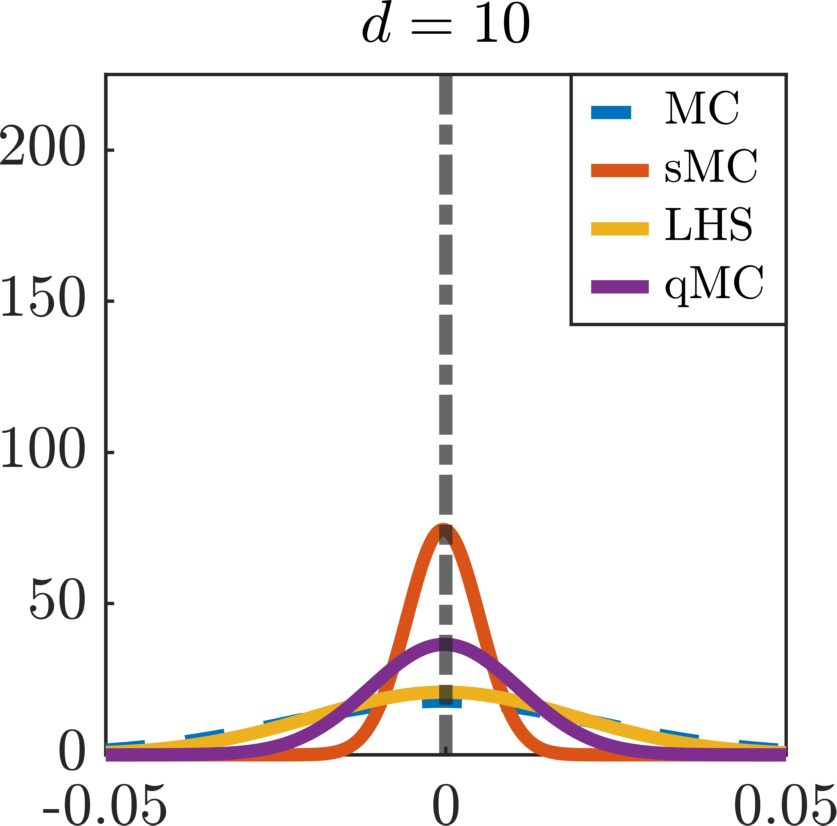} &
\includegraphics{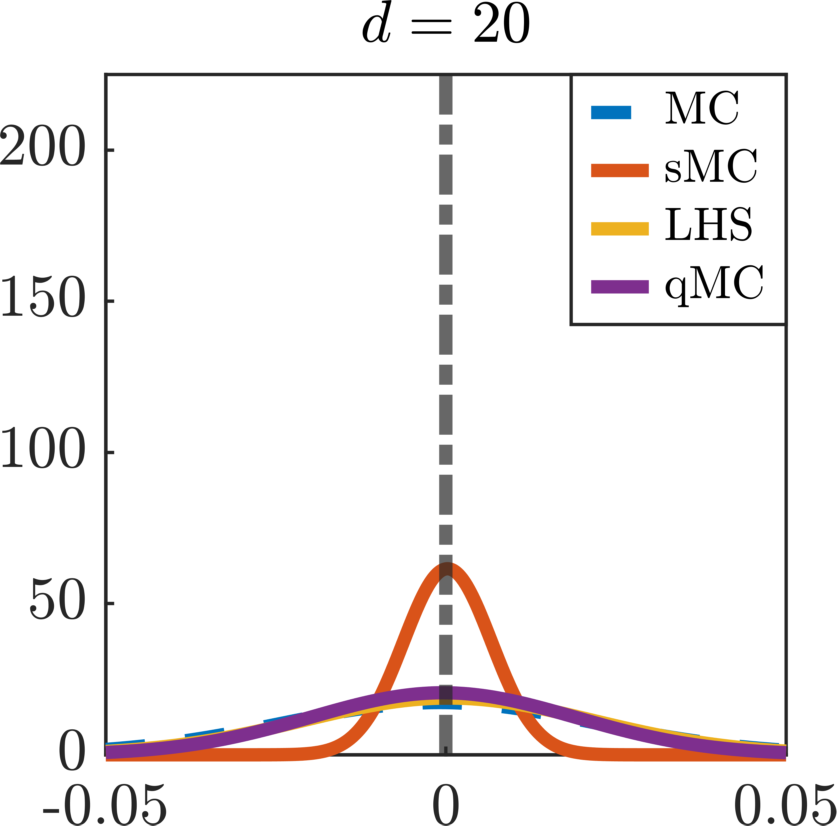} 
\end{tabular}
\end{center}
\vspace{0.25cm}
\begin{center}
\begin{tabular}{c|ccc}
\toprule
& $d = 5$ & $d = 10$ & $d = 20$ \\
\midrule
MC & 4.68e-04 (1) & 5.05e-04 (1) & 5.33e-04 (1) \\
sMC & 1.74e-05 (3.72e-02) & 2.88e-05 (5.70e-02) & 4.20e-05 (7.89e-02) \\
LHS & 1.43e-04 (3.05e-01) & 3.70e-04 (7.32e-01) & 4.51e-04 (8.47e-01) \\
qMC & 2.88e-06 (6.15e-03) & 1.19e-04 (2.36e-01) & 3.77e-04 (7.09e-01) \\
\bottomrule
\end{tabular}
\end{center}
\vspace{0.25cm}
\begin{center}
\begin{tabular}{ccc}
\includegraphics{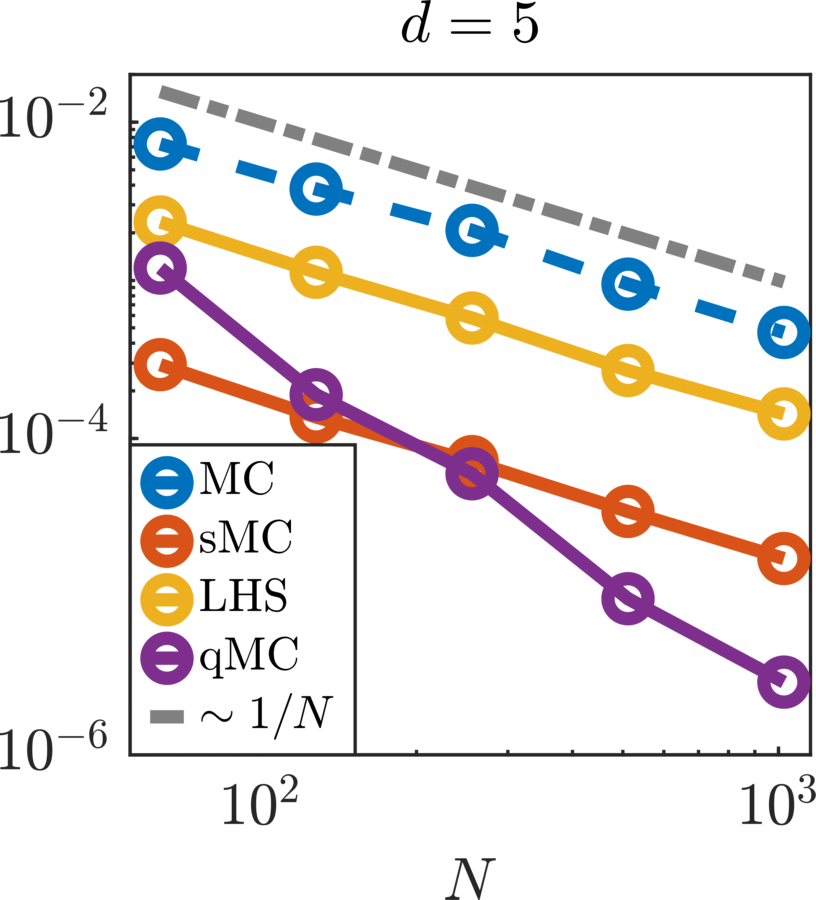} &
\includegraphics{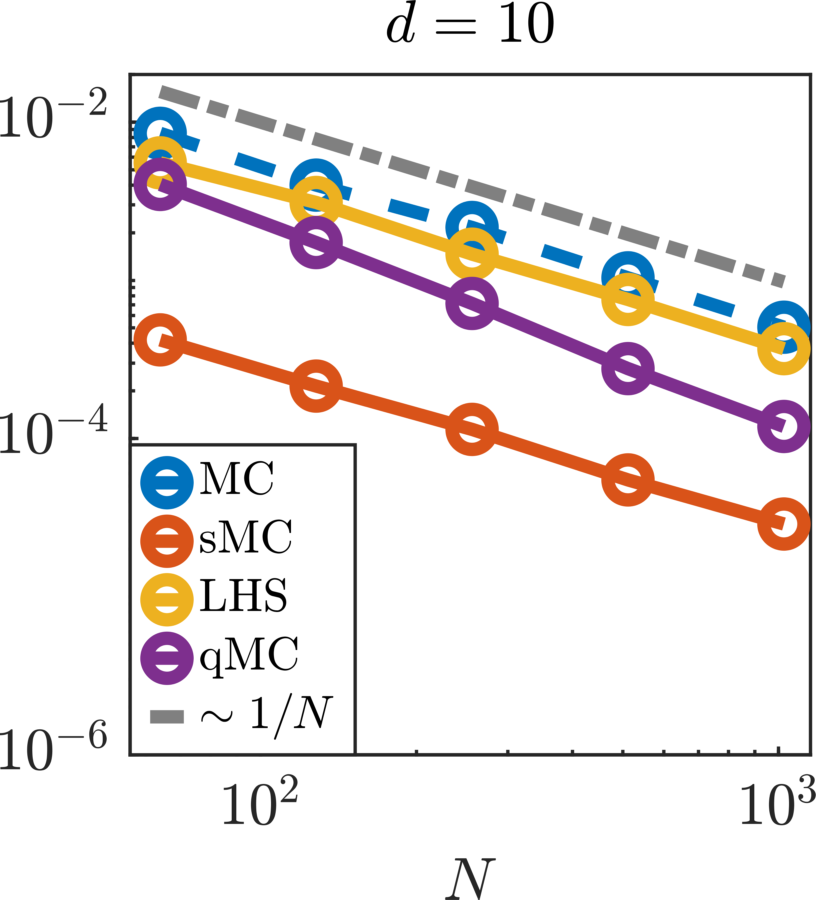} &
\includegraphics{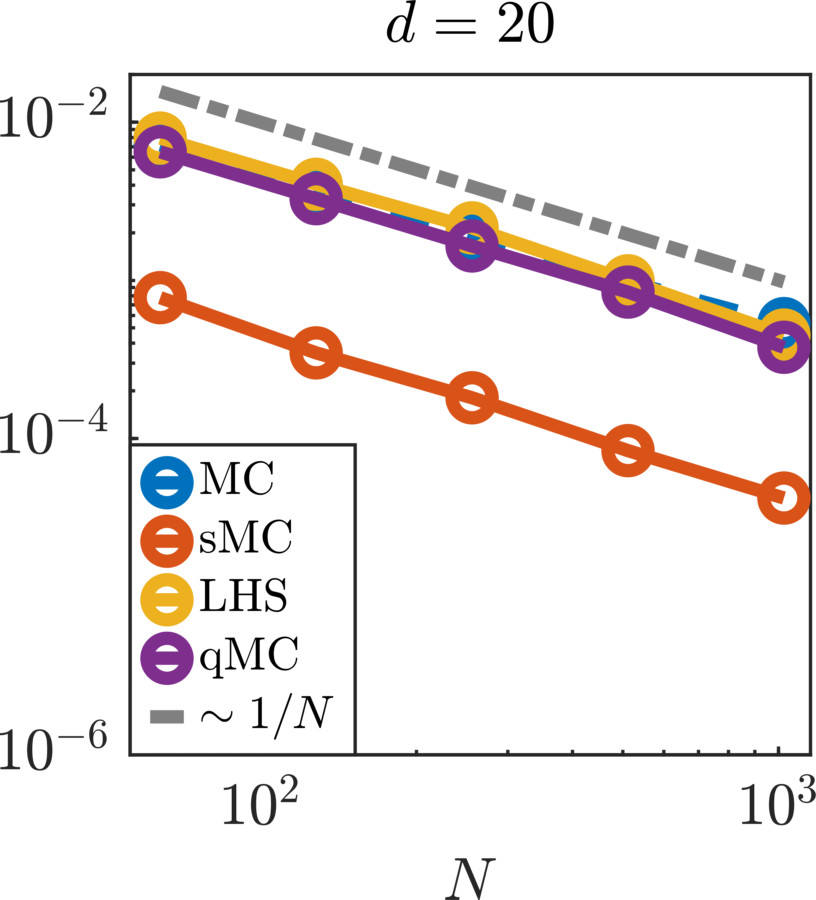}
\end{tabular}
\end{center}
\caption{\emph{TOP}: Comparison of standard Monte Carlo $\widehat q_\MC$ (dashed line) and the proposed NeurAM-based stratified estimator $\widehat q_\sMC$ with Latin Hypercube Sampling $\widehat q_\mathrm{LHS}$ and Quasi-Monte Carlo using Sobol sequences $\widehat q_\mathrm{qMC}$ (solid line), for different dimensions $d = 5, 10, 20$. The gray dash-dotted vertical line represents the exact value for the expected value of the model's output. \emph{MIDDLE}: Mean squared error (MSE) of the estimators obtained with $N=1024$. The numbers in parentheses denote the ratio between the MSE of each estimator and the Monte Carlo reference value. \emph{BOTTOM}: MSE as a function of the computational budget $N$.}
\label{fig:comparison}
\end{figure}

Different sampling-based variance reduction techniques have been proposed in the literature beyond stratified sampling. In this section, we compare our proposed methodology for stratification with two well-established strategies: Latin Hypercube Sampling (LHS) and randomized Quasi-Monte Carlo (qMC). We note that deterministic qMC is not considered here, as our focus is on variance reduction for unbiased estimators and, by construction, deterministic qMC does not possess variance. It is well-known that both LHS and randomized qMC significantly outperform standard MC and achieve higher convergence rates for smooth models in low dimensions. Also, depending on the specific application, LHS and randomized qMC can have different advantages. For instance, with LHS (and stratification) it is possible to generate sets of samples with arbitrary number of points, whereas for qMC (either deterministic or randomized) the convergence properties are guaranteed only if the number of samples corresponds to a power of two, which can be overly restrictive for many applications. Moreover, for both LHS and qMC, the dimensionality of the problem has an effect on their performance. To illustrate the effect of dimensionality, we introduce an analytical model that can be easily scaled in term of its dimension. Consider the model $\Q^{(d)} \colon \R^d \to \R$ defined as
\begin{equation}
\Q^{(d)}(x) = \sin \left( \sum_{i=1}^d x_i \right),
\end{equation}
where $d$ denotes the input dimensionality, and assume a uniform input probability distribution $\mu = \mathcal U([-1,1]^d)$, which implies $q_d = \Ex^\mu[\Q_d(X)] = 0$. We then set $M = 1000$ and $K = 10^6$ to learn the NeurAM and the CDF, respectively. In \cref{fig:comparison}, for different values of the dimension $d = 5, 10, 20$, we compare independent repetitions of our stratified estimator with $S = 16$ strata to the same number of independent repetitions for standard Monte Carlo, LHS, and randomized qMC, assuming a computational budget $N = 1024$. Both LHS and qMC are implemented using the \texttt{scipy.stats.qmc} Python package. We also note that the randomized version of qMC Sobol' sequences in \texttt{scipy} is based on digital shift and includes linear matrix scrambling. Similar results can be obtained using randomized digital nets from the \texttt{QMCPy} Python package, whereas randomized Halton sequences produce results with higher variance. For all estimators, in \cref{fig:comparison} we also report the mean squared error (MSE) for $N = 1024$ and plot it as a function of an increasing computational budget $N = 64, 128, 256, 512, 1024$ (we limit the computational budget to powers of two to avoid penalizing qMC). While LHS performs overall only slightly better than standard Monte Carlo, randomized qMC provides a substantial improvement over both Monte Carlo and our stratified estimator in terms of MSE and convergence rate for $d = 5$. Similar results can also be observed for the functions used in the previous sections due to their dimensionality. However, qMC variance reduction compared to MC degrades as the dimension increases to $d = 10$ and offers no improvement for $d = 20$. In contrast, the NeurAM-based stratified Monte Carlo estimator is not significantly affected by increasing dimensionality, resulting in a more robust estimator that consistently reduces the variance of standard Monte Carlo across all three considered dimensions. Finally, we also remark that the evaluation of the estimator variance, which is needed for the construction of a confidence interval, is straightforward for Monte Carlo and stratified estimators (see, e.g., equation \cref{eq:sMC_general_NeurAM} for our stratified estimator), but it requires repetitions for qMC, which reduces the effective computational budget available for estimating the statistics of interest.

\subsection{Darcy flow} \label{sec:Darcy}

\begin{figure}
\begin{center}
\includegraphics{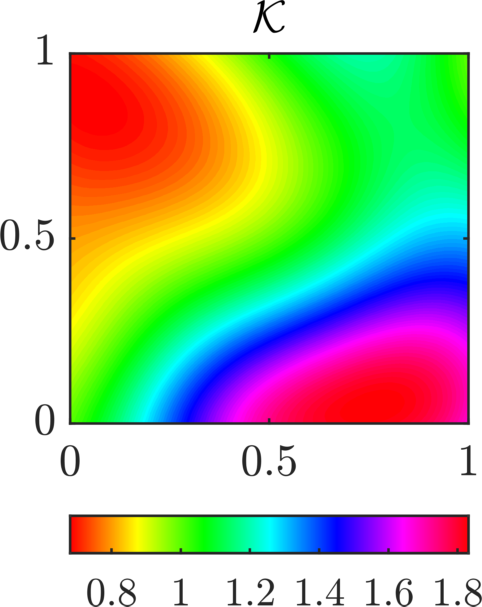}
\includegraphics{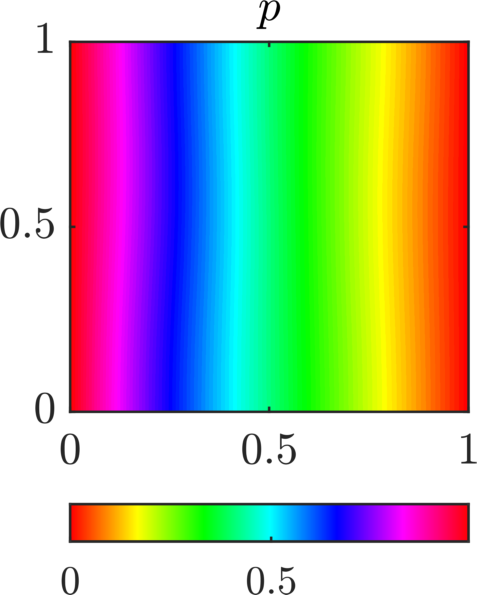}
\includegraphics{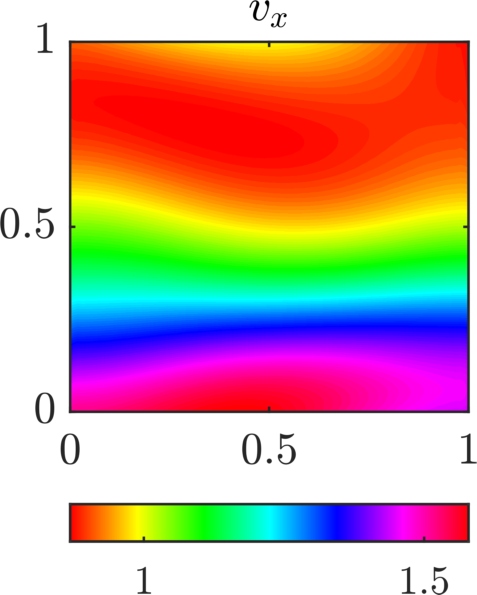}
\includegraphics{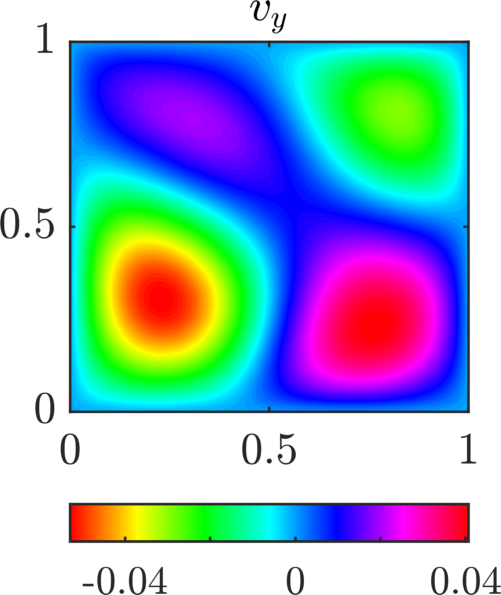}
\end{center}
\vspace{0.25cm}
\begin{center}
\begin{tabular}{cc}
\includegraphics{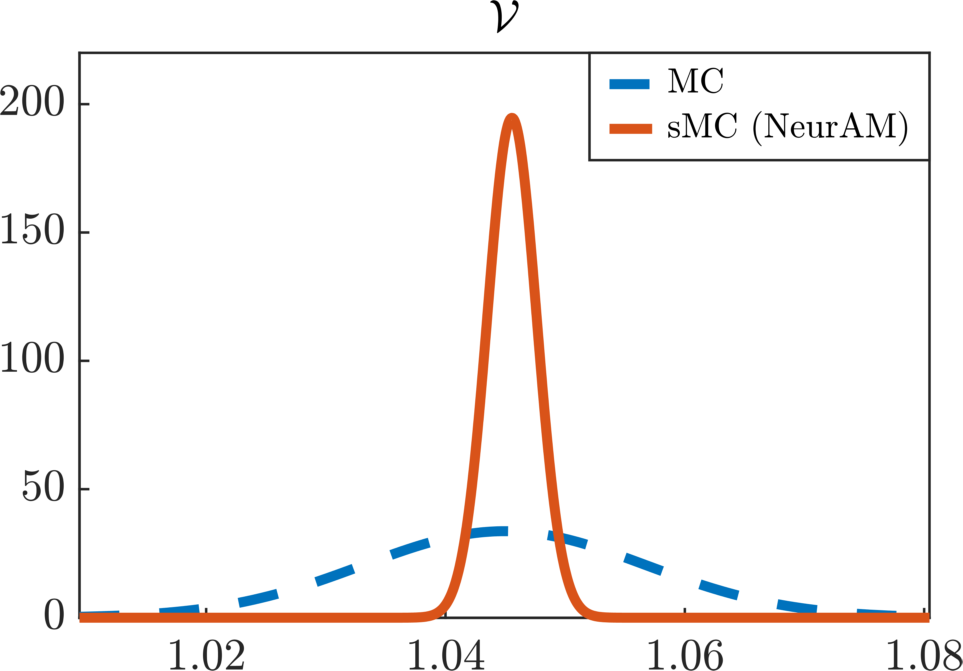} & \includegraphics{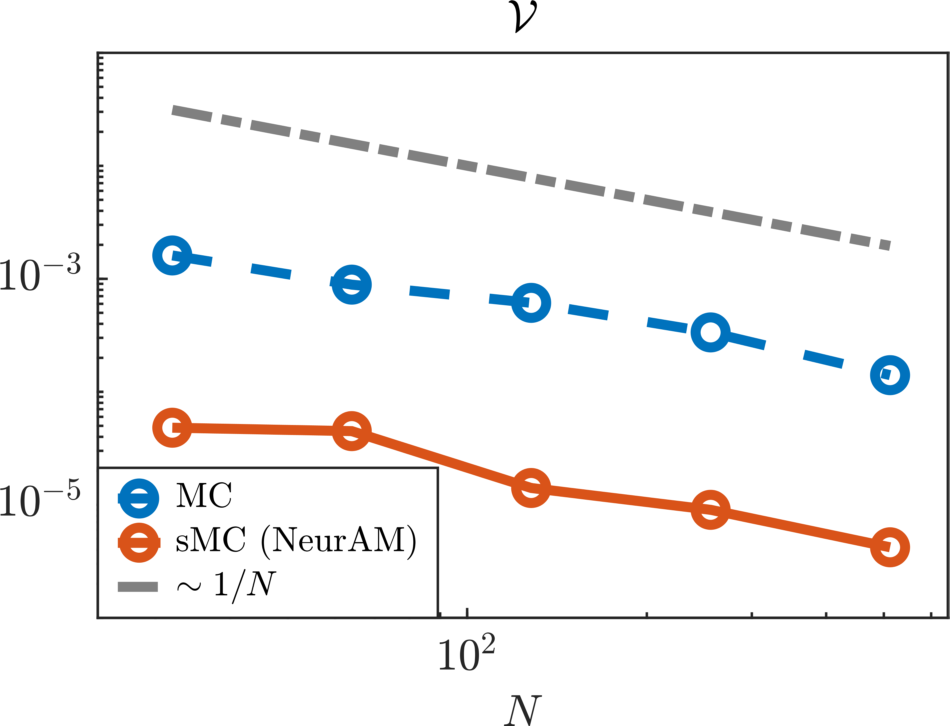}
\end{tabular}
\end{center}
\caption{\emph{TOP:} A single realization of the permeability $\mathcal K$, $p$, $v_x$, $v_y$ for the Darcy flow problem. \emph{BOTTOM-LEFT:} Comparison between standard Monte Carlo estimator $\widehat q_\MC$ (dashed line) and NeurAM-based stratified estimator $\widehat q_\sMC$. \emph{BOTTOM-RIGHT:} Variance as a function of the computational budget $N$.}
\label{fig:Darcy}
\end{figure}

In the last numerical experiment, we consider the two-dimensional Darcy flow problem, which models a single phase, steady-state flow through a random permeability field. Let $\Omega = [0,1]^2$ be the spatial domain representing an idealized oil reservoir, and let $p \colon \Omega \to \R$ and $v \colon \Omega \to \R^2$ be the pressure and velocity field of the fluid, respectively, which satisfy the equations
\begin{equation}
\begin{aligned}
- \nabla \cdot (\mathcal K \nabla p) &= f, &&\text{in } \Omega, \\
v &= - \mathcal K \nabla p, &&\text{in } \Omega,
\end{aligned}
\end{equation}
where $\mathcal K \colon \Omega \to \R_+$ is the permeability field and $f \colon \Omega \to \R$ is a source term. Moreover, let $\Gamma_D^l = \{ 0 \} \times [0,1]$, $\Gamma_D^r = \{ 1 \} \times [0,1]$, $\Gamma_N = \partial \Omega \setminus (\Gamma_D^l \cup \Gamma_D^r)$, and impose the following Dirichlet and Neumann boundary conditions
\begin{equation}
\begin{aligned}
p &= 1, &&\text{on } \Gamma_D^l, \\
p &= 0, &&\text{on } \Gamma_D^r, \\
\nabla p \cdot \mathrm n &= 0, &&\text{on } \Gamma_N.
\end{aligned}
\end{equation}
The random permeability $\mathcal K$ is modeled as $\mathcal K = e^G$, where $G$ is a Gaussian random field $G \sim \mathcal N(0, \Xi)$ with Gaussian covariance function
\begin{equation}
\Xi((x_1,y_1), (x_2,y_2)) = \tau^2 \exp \left[ - \frac{(x_1 - x_2)^2 + (y_1 - y_2)^2}{\ell^2} \right].
\end{equation}
In order to obtain finite-dimensional samples of $G$, we employ the truncated Karhunen–-Loève expansion with $\kappa$ terms
\begin{equation}
G(x,y) = \sum_{j=1}^\kappa \xi_j \sqrt{\theta_j} \psi_j(x,y),
\end{equation}
where $(\psi_j)_{j=1}^\infty$ are eigenfunctions of the covariance operator $\Xi$ and form an orthonormal basis of $L^2(\Omega)$, $(\theta_j)_{j=1}^\infty$ are the corresponding eigenvalues, and $(\xi_j)_{j=1}^\kappa$ are independent and identically distributed standard Gaussian random variables, i.e., $\xi_j \sim \mathcal N(0,1)$. Our goal is then estimating the expectation of the quantity of interest $\mathcal V$ given by
\begin{equation}
\mathcal V = \int_0^1 \int_0^1 (v_x(x,y)^2 + v_y(x,y)^2) \dd x \dd y,
\end{equation}
where $v_x$ and $v_y$ denote the two components of the velocity field $v$. 

The Darcy flow simulations are performed using the FEniCS computing platform \cite{LMW12} with a spatial discretization of $64$ elements per dimension. A representative realization of the permeability field $\mathcal K$, the pressure $p$, and the velocity components $v_x$ and $v_z$ is shown in \cref{fig:Darcy}. In this experiment, we set the source term to zero, $f = 0$, choose a truncation level $\kappa = 100$, and fix the standard deviation and correlation length of the covariance function to $\tau = 0.2$ and $\ell = 0.5$, respectively. The lower part of the same figure compares the performance of the NeurAM-based stratified estimator with that of a standard Monte Carlo estimator. We employ $M = 2000$ samples to train the NeurAM, $K = 10^6$ for the CDF, and assume a total computational budget of $N = 512$. Unlike the previous test cases, due to the increased cost of each simulation, we employ 20 independent evaluations of each estimator for variance estimation. The resulting variances are $\Var[\widehat q_\MC] = \text{1.40e-4}$ and $\Var[\widehat q_\sMC] = \text{4.20e-6}$, yielding a variance ratio of 3.00e-2. In addition, we report the estimator variance as a function of the computational budget $N = 32, 64, 128, 256, 512$. These results, obtained for random inputs in $d = \kappa = 100$ dimensions, show that the proposed stratification approach significantly reduces Monte Carlo variance even in high-dimensional settings.

\section{Conclusion}\label{sec:conclusion}

Stratified sampling is a well-known variance reduction technique in Monte Carlo estimation.
Despite its good performance in simple test cases, it does not scale well to high dimensions, making it inefficient for real applications.
To overcome this problem, we combine stratification on the unit interval with NeurAM, a recently proposed data-driven strategy for nonlinear dimensionality reduction. 
NeurAM generates strata that are adapted to the variation of the underlying model and generally bounded by its level sets.
Our approach is easy to implement, and can be effectively combined with other techniques for variance reduction, for example multifidelity estimators. 
We analyze this latter combination in detail, showing that our approach leads to consistent variance reduction.
We study the conditions leading to optimal stratification and optimal sample allocation for both stratified and multifidelity stratified estimators, and provide a heuristic algorithm to sequentially refine a collection of strata, that iteratively reduces the variance of the resulting estimator. 
Moreover, we demonstrate the performance of the proposed stratified estimators through several numerical experiments obtaining effective variance reduction for both low- and high-dimensional problems. 

The idea presented in this work could be extended in multiple directions. First, we applied NeurAM-based stratification to multifidelity Monte Carlo estimators, but other variance reduction strategies could also be considered.
In addition, similarly to what we do in \cref{sec:num_AS} for active subspaces, alternative techniques for dimensionality reduction could be leveraged to improve stratified sampling. 
One limitation of the current approach, inherited from the NeurAM methodology \cite{ZGS25}, concerns its performance for models exhibiting discontinuities and bifurcations in which the low-dimensional manifold is not simply connected, and addressing this limitation is an interesting direction for future work.
Moreover, even if, ideally, a one-dimensional manifold should be sufficient to capture the response surface of a scalar quantity of interest, for more complex problems, a NeurAM with more than one dimension could provide improved results. However, this direction involves several nontrivial challenges. 
First, by moving to higher dimensions, we lose the interpretation that the strata tend to follow the contour lines of the model, and it becomes unclear how to perform stratification in the latent space so that it remains effective in the original domain.
Second, inverse transform sampling is no longer straightforward in dimensions greater than one. Hence, one would need a method to map the latent space distribution to a uniform distribution on the unit hypercube, e.g., via normalizing flows, flow matching, or diffusion models, which should be trained in addition to NeurAM. 
Third, the extension of NeurAM itself to a multi-dimensional latent space requires investigation. In particular, as discussed in \cite[Section 5]{ZGS25}, we would like to preserve some ordering or structure in the latent coordinates, similar to the active subspace method, where reduced variables are ranked based on their importance.
Finally, this study can be extended to higher-order statistical moments of a given quantity of interest, or for variance reduction in sensitivity analysis.

\subsection*{Acknowledgements}
We thank the Editor and the anonymous referees for their thorough reading and for their insightful comments and valuable suggestions, which have helped improve and clarify this manuscript. This article has been authored by an employee of National Technology \& Engineering Solutions of Sandia, LLC under Contract No. DE-NA0003525 with the U.S. Department of Energy (DOE). The employee owns all right, title and interest in and to the article and is solely responsible for its contents. The United States Government retains and the publisher, by accepting the article for publication, acknowledges that the United States Government retains a non-exclusive, paid-up, irrevocable, world-wide license to publish or reproduce the published form of this article or allow others to do so, for United States Government purposes. The DOE will provide public access to these results of federally sponsored research in accordance with the DOE Public Access Plan https://www.energy.gov/downloads/doe-public-access-plan.
DES acknowledges support from NSF CDS\&E award \#2104831, and NSF CAREER award \#1942662. AZ is supported by ``Centro di Ricerca Matematica Ennio De Giorgi'' and the ``Emma e Giovanni Sansone'' Foundation, and is member of INdAM-GNCS.
The authors would like to thank the Center for Research Computing at the University of Notre Dame for providing computational resources and support that were essential to generate the results for the Darcy flow problem.

\bibliographystyle{siamnodash}
\bibliography{biblio}

\end{document}